\documentclass[11pt]{amsart}

\usepackage{tikz}
\usepackage{graphicx}

\usepackage{amsmath}
\usepackage{tikz}
\usepackage{amssymb}
\usepackage{mathcomp,wasysym}

\numberwithin{equation}{section}
\newtheorem{theorem}{\sc Theorem}[section]
\newtheorem{lemma}{\sc Lemma}[section]

\newtheorem{remark}{\sc Remark}

\def\R{\mathbb R}

\def\n{\nonumber}

\def\p{\tilde p}

\def\f{\frac{n-1}{n}}

\def\ost{\Omega^s(t)}
\def\oft{\Omega^f(t)}
\def\oftp{\Omega^-(t)}
\def\ostp{\Omega^+(t)}
\def\oftn{\Omega^f(t_n)}

\def\p{\partial}
\def\n{\nonumber}
\def\f{\frac}
\def\ttau{\tilde{\tau}}
\def\tn{\tilde{n}}

\title[Finite-time singularity formation for moving interface Euler equations]
      {Finite-time singularity formation for moving interface Euler equations}

\author[D. Coutand]{Daniel Coutand}
\address{Maxwell Institute for Mathematical Sciences and department of Mathematics,
Heriot-Watt University, Edinburgh, EH14 4AS, UK}
\email{d.coutand@hw.ac.uk}
     
\date{\today}    

\begin{document}

\maketitle

\begin{abstract}
This paper provides a general method for establishing finite-time singularity formation for moving interface problems involving the incompressible Euler equations in the plane. This methodology is applied to two different problems.

The first problem considered is the two-phase vortex sheet problem with surface tension, for which, under suitable assumptions of smallness of the initial height of the heaviest phase and velocity fields, is proved the finite-time singularity of the natural norm of the problem. This is in striking contrast with the case of finite-time splash and splat singularity formation for the one- phase Euler equations of \cite{CaCoFeGaGo2013} and \cite{CoSh2014}, for which the natural norm (in the one-phase fluid) stays finite all the way until contact.

The second problem considered involves the presence of a heavier rigid body moving in the inviscid fluid for which well-posedness was recently established in \cite{GlSu2015}. For a very general set of geometries (essentially the bottom of the symmetric domain being a graph) we first establish that the rigid body will hit the bottom of the fluid domain in finite time. This result allows for more general geometries than the ones first considered by \cite{MuRa2015} for this problem, as well as for small square integrable vorticity. Next, we establish the blow-up of a surface energy and a characterization of acceleration at contact: It opposes the motion,  and is either strictly positive and finite if the contact zone is of non zero length, or infinite otherwise.

\end{abstract}

\section{Introduction}
\label{intro}

Finite-time singularity formation in moving boundary problems have been an active field of research for at least the past 20 years. Historically the first problems studied were for a rigid body moving in a fluid (see \cite{Hi2007}, \cite{HiTa2009}, \cite{GVHi2010}, \cite{GVHiWa2015} for the viscous fluid case and \cite{HoMu2008}, \cite{MuRa2015} for the inviscid case), which presents the simplification at the level of the analysis of having a constant shape for the inclusion. More recently the case of one-phase and two-phase Euler interface problem have started to be considered. The present paper presents a new methodology addressing finite-time singularity formation for any type of problems when the fluid equations are the incompressible Euler equations and the physical law of the included phase provides spatial control of the position of the interface.

 The first problem considered is the formation of finite-time singularity for the two-phase moving interface Euler equations with surface tension. This problem is known to be locally in time well-posed for a natural norm $N(t)$ encoding the Sobolev regularity of the velocity field in each phase and the regularity of the moving interface (see \cite{Am2003}, \cite{AmMa2007} for the irrotational case, and \cite{ChCoSh2008}, \cite{ShZe2008},\cite{ShZe2011} for the case with vorticity).
 
 The one-phase water waves problem is known to be locally in time well-posed in Sobolev spaces, as the pressure condition holding in this situation avoids any Rayleigh-Taylor instability (\cite{Wu1997}, \cite{Wu1999}, \cite{La2005}, \cite{AlBuZu2014} for the case without vorticity and \cite{ChLi2000}, \cite{Li2004}, \cite{CoSh2007}, \cite{ShZe2008a} for the case with vorticity). The first type of singularity formation in finite time for this problem in Sobolev spaces was established by Castro et al in \cite{CaCoFeGaGo2013} by introducing the notion of splash and splat singularity, which is the self-intersection of the moving free boundary while the curve remains smooth (but is no longer locally on one side of its boundary at contact). This result was generalised in $3$-D and with vorticity by 
Coutand and Shkoller \cite{CoSh2014} by a very different approach. Our approach can be easily applied to many one-phase hyperbolic free boundary problems. It is to be noted that this type of splash singularity is purely restricted to a loss of injectivity, since the natural norm of the problem stays bounded until the time of contact. 

A natural question that then arose was to extend this type of self-contact along a smooth curve in the two-phase context (with surface tension to make the problem locally well-posed in Sobolev spaces). With different methods, Fefferman, Ionescu and Lie \cite{FeIoLi2013} and Coutand and Shkoller\cite{CoSh2016},  established that the two-phase vortex sheet problem with surface tension does not have finite-time formation of a splash or splat singularity so long as the natural norm of the problem for the velocity field in one phase stays bounded.  The results of \cite{CoSh2016} and \cite{FeIoLi2013} however do not exclude such a loss of injectivity. If it was to occur, it would involve blow-up of the natural norm of the problem in both phases.

The present paper introduces a new methodology, based upon studying the motion of the center of gravity of one of the two phases, which provides a differential inequality for a surface energy introduced in the present paper. We here establish that under some symmetry assumptions at time zero, and with gravity effects, there will either be a loss of injectivity or the natural norm of the problem $N(t)$ will blow-up in finite time. In both cases of this alternative, we show the natural norm of the problem blows up.  This result is in striking contrast with splash and splat singularity formation for one-phase water-waves problem introduced in \cite{CaCoFeGaGo2013}, and treated with different methods in a more general context in \cite{CoSh2014}, where the natural norm $N(t)$ stays finite. This was essential in the analysis of these papers in order to establish the finite-time contact, as this ensures that the magnitude of the relative velocity between two parts of an almost self-intersecting curve coming towards each other will be in magnitude greater than some strictly positive quantity. Such an approach would be impossible here, as in the two-phase problem, any contact would involve the formation of a cusp, which would make impossible high order elliptic estimates.
 
We next turn our attention to the case of the motion of a rigid body in a domain where the bottom is a graph. This is a simpler problem given that the shape of the interface stays constant for all time, which removes some considerable level of difficulty from the previous problem. The interest of this problem resides in allowing a more precise description of the behaviour at the time of singularity than for the case with deformable interface.

Recently, Glass $\&$ Sueur \cite{GlSu2015} proved that the motion of a rigid body in an inviscid fluid in a  domain in the plane is globally in time well posed so long as no contact occurs between the moving rigid body and the boundary. The qualitative question of whether contact singularity formation in finite-time is possible in the natural case where $u\cdot n=0$ on $\p\Omega$ arises then naturally.

The first results for finite-time contact for the rigid body case with zero vorticity in the inviscid fluid were obtained by Houot and Munnier \cite{HoMu2008} for the case of the disk in the half plane, and generalized by Munnier and Ramdani \cite{MuRa2015}, where they establish for the flat bottom case with the rigid body being a graph of the type $x_2=C |x_1|^{1+\alpha}$ ($\alpha>0$), or various combinations of bottoms and domains with specific concavity, that finite-time contact occurs, with a velocity which is shown to be either zero or non zero depending on $\alpha$. Other cases involving discussions on concavity of domains are also treated in \cite{MuRa2015}. It is to be noted that their methods, purely elliptic in nature in some rescaled infinite strip, require the zero vorticity assumption of their paper at the level of the rescaling of the elliptic problem in this infinite strip. 

For this problem of the rigid inclusion (where the shape of the inclusion does not change), our new methodology (based on a differential inequality for a surface energy that we identify) allows us to consider more general geometries than in \cite{MuRa2015} (namely we just need the bottom of the domain being a general graph, and no assumption for the solid, except symmetry with respect to the vertical axis) and also allow for small square integrable vorticity. We first establish here the question of finite-time contact at $T_{max}>0$ when gravity effects are taken into account (in particular the rigid body is assumed of higher density than the fluid phase). We then establish a set of blow-up properties satisfied by the fluid velocity and pressure fields and acceleration as $t\rightarrow T_{max}$, which are new for this kind of problems:\hfill\break
  $\bullet$ First, although the solid velocity stays bounded for all time of existence, the present paper establishes the fluid has a radically different behaviour, as the $L^2(\p\Omega)$ norm of the fluid velocity approaches $\infty$ near contact. This happens in a neighborhood of the contact zone, whereas away from the contact zone, the fluid velocity stays bounded. \hfill\break
 $\bullet$ Second, this work also establishes that the acceleration of the rigid body becomes infinite in the upward direction at the time of contact, except for the case where the contact zone contains a curve of non zero length, in which case the acceleration remains strictly positive and bounded close to the time of contact.  This behaviour is strikingly different from the behaviour of a material point falling in void (the basic question of elementary Newtonian mechanics), for which the motion has constant negative acceleration $-g$.
\hfill\break

 The plan of this paper is as follows. 
 
 In Sections 2 to 6, we remind the vortex sheet problem with surface tension, precise notations, and our type of initial data. In Section 7, we derive our differential inequality for a surface energy. This differential inequality structure appears by tracking the motion of the center of gravity of the heavier phase, using some symmetries of the Euler equations, and some elliptic estimates away from the heavier phase. It appears in a way quite natural to the problem of a moving Euler phase, and is also quite different from the pioneering works of Sideris \cite{Si1986} and Xin \cite{Xi1998} for compressible Euler and Navier-Stokes equations. We then use this differential inequality to establish in Section 8 our first theorem on finite-time singularity formation:
  \begin{theorem}
 \label{theorem1}
 Let $\Omega$ be a symmetric domain with respect to $x_1=0$ of class $H^{\frac{9}{2}}$ localy on one side of its boundary and satisfying the assumptions of Section \ref{section6} (in particular the bottom of $\p\Omega$ is a general graph), and let $\overline{\Omega^+}\subset\Omega$ be a domain of same regularity whose center of gravity is at altitude $h$ at time zero and symmetric with respect to $x_1=0$. Let $\Omega^-=\Omega\cap{\overline{\Omega^+}}^c$. Let us assume that $$\|u^\pm_0\|_{L^2(\Omega^\pm)}+\|\omega_0^-\|_{L^2(\Omega^-)}+h+ |\p\Omega^+|$$ is small enough, and that $\rho^+>\rho^-$. Then, for some $T_{max}\in (0,\infty)$:\hfill\break
 \noindent 1) either $\displaystyle\lim_{t\rightarrow T_{max}^-} N(t)=\infty$, where
 $$N(t)=\|u^-\|_{H^3(\oftp)}+\|u^+\|_{H^3(\ostp)}+\|\eta^f\|_{H^{\f{9}{2}}(\Omega^-)}\,,$$ where $\eta^-$ denote the Lagrangian flow map associated to $u^-$,\hfill\break
 \noindent 2) or there is either a self-intersection of the interface $\p\Omega^+(T_{max})$ with itself, or contact of $\p\Omega^+(T_{max})$ with $\p\Omega$ at $T_{max}$,\hfill\break
 \noindent 3) or $\int_0^{T_{max}}\int_{\p\Omega} |u^-|^2\ dldt=\infty\,.$
 \end{theorem}
 
\begin{remark}
 If we were to assume the density of the material initially inside $\Omega^+$ to be strictly smaller than the density of the material initially in $\Omega^-$, similar theorems would hold, assuming the top of $\p\Omega$ to be under the form of a graph.
 \end{remark}
 
 The cases 1) and 3) obviously involve a blow-up of $N$. We now show in Section 9 that the second case (corresponding to a loss of injectivity) leads to the blow-up of the following norm:

\begin{theorem}
\label{theorem2}
If the case 2) of Theorem \ref{theorem1} is satisfied then
\begin{equation}
\label{criteria}
\lim_{t\rightarrow T_{max}^-} [\|\nabla_\tau\tau\|_{L^\infty(\Gamma(t))}+\int_0^t \|\nabla u^-\|_{L^\infty(\Omega^-(t))}]=\infty\,,
\end{equation}
where $\tau$ denotes the unit tangent on $\p\Omega^+(t)$.
\end{theorem}
 
\begin{remark} 
We therefore have proved that in finite time $N(t)$ blows up in finite time for all situations of Theorem 1.
\end{remark}

 We next consider the case of the rigid body in an inviscid fluid. For this problem we show a small curl guarantees a monotone fall simply by conservation of energy, whereas in the vortex sheet problem with surface tension, there is no guarantee that the fall even occurs (locally in time when the solution is smooth, it can be guaranteed, but not as the singularity forms). In Section 11, the problem is reminded. In Section 12, it is shown the stream function satisfies a specific structure for this problem. In Section 13, the velocity is shown to be non zero before contact and in Section 14, finite-time contact is established:

\begin{theorem}
\label{theorem3}
Let $\Omega$ and $\Omega^s$ be $C^1$ domains satisfying the assumptions of Section 11 (essentially the part of $\p\Omega$  where contact potentially occurs is a graph). Let us assume that $v^s(0)=(0,v^s_2)$, with $v^s_2< 0$, and that $\rho^f<\rho^s$. Let us furthermore assume that the odd (with respect to $x_1$) vorticity satisfies 
\begin{equation}
\label{smallvortex}
\|\omega_0\|^2_{L^2(\Omega^f)}< \min\left(\f{m_s}{\rho^fD_\Omega} |v^s_2(0)|^2,\f{(m_s-\rho^f|\Omega^s|)g}{2\rho^f (CD_\Omega+C)}\right)\,,
\end{equation}
 with $C_\Omega$ being the standard Poincar\'e constant in $H^1_0(\Omega)$, given by (\ref{1403.3}), $D_\Omega$ being given by (\ref{040717.1}) and $C$ being given by (\ref{reg1}). 

 Then there exists $T_{max}\in (0,\infty)$ such that the rigid body will touch $\p\Omega$ at time $T_{max}$ with a finite velocity $v^s_2(T_{max})\le 0$.
 
\end{theorem}

   The next sections are for $\omega=0$. In Section 15, is established an essential comparison of various norms of the velocity in the fluid by elliptic techniques proper to this problem. In Section 16, we provide a simpler formula for acceleration. In Section 17 is established the blow-up of the  $L^2(\p\Omega)$ norm and in Section 18 is proved the positive or infinite character of acceleration at contact depending on the contact zone.
\begin{theorem}
\label{theorem4}
Let us assume furthermore that $\omega=0$ and that $\Omega$ and $\Omega^s$ are $C^2$. Then, with $T_{max}$ obtained in Theorem \ref{theorem1}, we have the following properties:\hfill\break
\noindent 1) $\displaystyle\lim_{t\rightarrow T_{max}^-}\|u^f\|_{L^2(\p\Omega)}=\infty$,\hfill\break
\noindent 2) $$\displaystyle\lim_{t\rightarrow T_{max}^-}\left|\int_{\p\ost} p n \ dl\right|=\lim_{t\rightarrow T_{max}^-}\f{dv^s_2}{dt}=\infty\,,$$ except for the case where the contact zone between $\p\Omega$ and $\p\Omega^s(T_{max})$ contains a curve of nonzero length, in which case we have $$0<\liminf_{t\rightarrow T_{max}^-} \f{dv^s_2}{dt}\le\limsup_{t\rightarrow T_{max}^-} \f{dv^s_2}{dt}<\infty\,.$$
\end{theorem}
 \begin{remark}
 Point 2) shows a drastic difference between the problem of the rigid body in an inviscid fluid and in void, since in the case with void, the acceleration remains constant ($=-g$) for all time even at contact. It shows that the rigid body does feel the imminence of contact and tries to avoid it by an upward acceleration (finite or infinite according to the size of the contact zone) opposing the fall.
 \end{remark}
 
 \begin{remark}
 Any physical model such that $\p\Omega^+(t)$ would have a position controlled in $L^\infty$ norm would be suitable for this theory. 
 \end{remark}

 \section{The vortex sheet problem with surface tension}
\label{section2}

 The vortex sheet problem with surface tension is a moving interface problem  locally in time well-posed from \cite{Am2003}, \cite{AmMa2007} for the irrotational case, and \cite{ChCoSh2008}, \cite{ShZe2008},\cite{ShZe2011} for the case with vorticity. 

Here, $\Omega\subset\mathbb{R}^2$ is a smooth bounded domain of class $H^{\f{9}{2}}$, locally on one side of its boundary,  and $\Omega^+\subset\Omega$ is also a smooth bounded domain of class $H^{\f{9}{2}}$, locally on one side of its boundary, and such that $\overline{\Omega^+}\subset\Omega$. We consider the incompressible Euler equations for the motion of two fluids of densities $\rho^-$ and $\rho^+$ that are at time zero in $\Omega\cap\overline{\Omega^+}^c=\Omega^-$ and $\Omega^+$, with surface tension and gravity effects:
\begin{subequations}
\label{ff1}
\begin{align}
\rho^{\pm} (u_t^\pm + u^\pm\cdot\nabla u^\pm)+\nabla p^\pm=&-\rho^\pm g e_2\,,\ \text{in}\ \Omega^\pm(t)\label{ff1.a}\\
\text{div}\ u^\pm=&0\,,\ \text{in}\ \Omega^\pm(t)\,,\label{ff1.c}\\ 
(p^--p^+)\ n=&-\sigma \nabla_{\tau}(\tau)\,,\ \text{on}\ \p\Omega^+(t)\,,\label{ff1.e}\\
u^-\cdot n=& u^+\cdot n\,,\ \text{on}\ \p\Omega^-(t)\,, \label{ff1.f}\\
u^-\cdot n=& 0\,,\ \text{on}\ \p\Omega\,, \label{ff1.g}\\
\Omega^\pm(0)=&\Omega^\pm\,,\\
u^\pm(x,0)=& u_0^\pm\,,\ \text{in}\ \Omega^\pm\,.
\end{align}
\end{subequations}
where the material interface $\p\Omega^-(t)$ moves with speed $u^-\cdot n=u^+\cdot n$, where $n$ is the outward unit normal to $\oftp$, $\tau$ is the unit tangent, and $e_2$ is the unit vertical vector pointing upward. Also the surface tension coefficient $\sigma$ is classically assumed strictly positive.

If $\eta^\pm$ denote the Lagrangian flow map associated to $u^\pm$, defined by
\begin{align*}
\eta^\pm_t(x,t)=&u^\pm(\eta^\pm(x,t),t)\,,\ \forall x\in \Omega^\pm\,, t\ge 0\,,\\
\eta^\pm(x,0)=&x\,,
\end{align*} we showed in \cite{ChCoSh2008} that the problem is locally in time well-posed for the norm:
\begin{equation}
\label{ff2}
N(t)=\|\eta^-\|_{H^{\f{9}{2}}(\Omega^-)}+\|u^-\|_{H^3(\Omega^-(t))}+\|u^+\|_{H^3(\Omega^+(t))}\,.
\end{equation}

We also define the Lagrangian velocity $v^\pm(x,t)=u^\pm(\eta^\pm(x,t),t)$. 

We will show this problem has a finite-time singularity formation
provided some assumptions are made on the initial domain and data and that $$\rho^+>\rho^-\,.$$ To this end, we will establish that if $N(t)$ stays finite for all time, then either a finite in time contact occurs (either self intersection of $\p\ostp$ or between $\p\ostp$ and $\p\Omega$) or a surface energy blows up. In case of contact, we will show this leads to the blow-up  (\ref{criteria}).

\subsection{Global vector field in $\overline{\Omega}$ extending the normal}

We will need later on a smooth vector field extending the normal to $\p\Omega$ into $\Omega$.

We denote by $n$ the outward unit normal to $\Omega$, and by $\tilde n$ the smooth solution of the elliptic problem:
\begin{subequations}
\label{ch0}
\begin{align}
\triangle \tilde n &=0\,,\ \text{in}\ \Omega\,,\\
\tilde n&=n\,,\ \text{on}\ \partial\Omega\,.
\end{align}
\end{subequations}
By the maximum and minimum principles we have that for each component of $\tilde n$,
\begin{equation}
\label{maximumprinciple}
|\tilde n_i|\le 1\,.
\end{equation}

Given the regularity of $\p\Omega$, we have by elliptic regularity that $\tilde n\in H^3(\Omega)\subset C^{1}(\overline{\Omega})$. 
Therefore,
\begin{equation}
\label{choice2}
\|\nabla \tilde n\|_{L^\infty(\Omega)}\le \beta_{\Omega}<\infty\,.
\end{equation}

We then define the vector field $\tilde\tau=(\tilde n_2,-\tilde n_1)\in H^3(\Omega)$, which extends the tangent to $\p\Omega$ inside $\Omega$.

\section{notations}
\label{section3}
We have $n=(n_1,n_2)$ denote the outer unit normal to $\oftp$, and $\tau=(\tau_1,\tau_2)=(n_2,-n_1)$ denote the unit tangent vector field.

For a smooth domain $A\subset{\R}^2$ we denote by $|A|$ its area and by $|\p A|$ the length of its boundary.

Due to incompressibility, we have for all time of existence $|\Omega^\pm(t)|=|\Omega^\pm|$.  

We also use the Einstein convention of summation with respect to repeated indices and exponents.

 For a given vector $a\in\R^2$, we denote $\displaystyle\nabla_a u=a_i\f{\p u}{\p x_i}\,.$ Of particular interest will be the case when either $a=\tau(x)$, or $a=n(x)$. In that case the divergence of a vector field $u$ written in the $(\tau(x),n(x))$ basis instead of the $(e_1,e_2)$ basis reads as:
\begin{equation}
\label{3.1}
\operatorname{div}u=(\nabla_{\tau(x)} u)\cdot\tau(x)+(\nabla_{n(x)} u)\cdot n(x)\,,
\end{equation}
while the curl reads as:
\begin{equation}
\label{3.2}
\omega=\operatorname{curl}u=(\nabla_{\tau(x)} u)\cdot n(x)-(\nabla_{n(x)} u)\cdot\tau(x)\,.
\end{equation}

Another context in which these derivatives will be encountered is integration along closed curves. For instance, if $\theta$ is a smooth $1-$ periodic parameterization of a closed curve $\gamma$, we have the following properties that will be used extensively:
\begin{equation}
\label{3.4}
\tau(\theta(s))=\f{\f{\p\theta}{\p s}}{\left|\f{\p\theta}{\p s}\right|}(s)\,,
\end{equation}
\begin{equation}
\label{3.3}
\f{\p(u\circ\theta)}{\p s}=\f{\p\theta_i}{\p s}\ \f{\p u}{\p x_i}\circ\theta=\left|\f{\p\theta}{\p s}\right|(\nabla_{\tau} u) (\theta(s))\,,
\end{equation}
\begin{equation}
\int_{\gamma}\nabla_{\tau} u\ dl=\int_0^1 \f{1}{\left|\f{\p\theta}{\p s}\right|} \f{\p(u\circ\theta)}{\p s}\ \underbrace{\left|\f{\p\theta}{\p s}\right|\ ds}_{=dl}=[u\circ\theta]_0^1=0\,.
\end{equation}

\vskip 1cm \section{Conservation of energy}
\label{section4}
For all time of existence it is classical that the quantity
\begin{equation}
\f{1}{2}\sum_{\pm} \rho^\pm \int_{\Omega^\pm(t)} |u^\pm(x,t)|^2\ dx+ \sum_{\pm}\rho^\pm g \int_{\Omega^\pm} \eta^\pm_2\ dx\n\\
+\sigma |\p\Omega^+(t)|\,, 
\label{ff4}
\end{equation}
is independent of time.
Now given that $$\int_{\Omega^\pm} \eta_2^\pm\ dx=\int_{\Omega^\pm(t)}x_2\ dx\,,\ \ \int_{\oftp}x_2+\int_{\ostp} x_2\ dx=\int_\Omega x_2\ dx\,,$$
we then infer from this conservation that the total energy
\begin{align}
E(t)=&\sum_\pm\f{\rho^\pm}{2}\int_{\Omega^\pm(t)} |u^\pm|^2\ dx+ (\rho^+ -\rho^-) g \int_{\ostp} x_2\ dx +\sigma |\p\Omega^+(t)|\n\\
=&\sum_\pm\f{\rho^\pm}{2}\int_{\Omega^\pm} |u^\pm(x,t)|^2 dx+ \underbrace{(\rho^+-\rho^-) }_{\ge 0} g \underbrace{x_2^+(t)}_{\ge 0} |\Omega^+|+\sigma|\p\Omega^+(t)|\,,
\label{ff5}
\end{align}
is constant in time for all time of existence of a smooth solution (namely so long as no eventual collision with the boundary occurs, or that no self-intersection of $\p\ostp$ occurs, and so long as the norm (\ref{ff2}) stays finite) and where we defined
\begin{equation}
\label{ff6}
x^+(t)=\frac{1}{|\Omega^+|}\int_{\ostp} x\ dx=\f{1}{|\Omega^+|}\int_{\Omega^+}\eta^+\ dx\,,
\end{equation}
as the center of gravity of $\ostp$. Tracking the motion of this center of gravity will prove a powerful tool in establishing our finite in time singularity formation result (since any pointwise estimate would be hopeless in a two-phase problem as a cusp forms in $\oftp$ at the time of contact). 

We then have for the velocity of the center of mass that:
\begin{equation}
\label{ff7}
v^+(t)=\frac{1}{|\Omega^+|}\int_{\Omega^+} v^+\ dx=\frac{1}{|\Omega^+|}\int_{\ostp} u^+\ dx\,,
\end{equation}

and for the acceleration:
\begin{equation}
\label{ff8}
a^+(t)=\frac{1}{|\Omega^+|}\int_{\Omega^+} \f{d v^+}{dt}\ dx=\frac{1}{|\Omega^+|}\int_{\ostp} u^+_t+u^+\cdot\nabla u^+\ dx\,,
\end{equation}

Due to (\ref{ff5}) and our definition (\ref{ff7}), we have by Cauchy-Schwarz
\begin{equation}
\label{ff9}
|v^+(t)|^2\le \frac{1}{|\Omega^+|}\int_{\ostp} |u^+|^2\ dx\le \f{2E(0)}{m_+}\,,
\end{equation}
where $m_+=\rho^+ |\Omega^+|$, 
which establishes the uniform in time control of this velocity.

\section{An equivalent formulation of the problem}
\label{section5}
First, using the definition of the curl, we see that:
\begin{equation*}
u^-_1\f{\p u^-_1}{\p x_1}+u^-_2\f{\p u^-_1}{\p x_2}=u^-_1\f{\p u^-_1}{\p x_1}+u^-_2\f{\p u^-_2}{\p x_1}-\omega^- u^-_2
= \f{1}{2}\f{\p |u^-|^2}{\p x_1}-\omega^- u^-_2\,.
\end{equation*}
Similarly,
\begin{equation*}
u^-_1\f{\p u^-_2}{\p x_1}+u^-_2\f{\p u^-_2}{\p x_2}=u^-_1\f{\p u^-_1}{\p x_2}+u^-_2\f{\p u^-_2}{\p x_2}+\omega^- u^-_1
= \f{1}{2}\f{\p |u^-|^2}{\p x_2}+\omega u^-_1\,.
\end{equation*}

Therefore, the Euler equations in $\oftp$ can be written as:
\begin{equation}
\label{e3bis}
\rho^- u^-_t+\nabla\left(\f{\rho^- |u^-|^2}{2}+p^-\right)=-\rho^- g e_2-\omega^- (-u^-_2,u^-_1)\,.
\end{equation}

\section{Choice of initial data}
\label{section6}

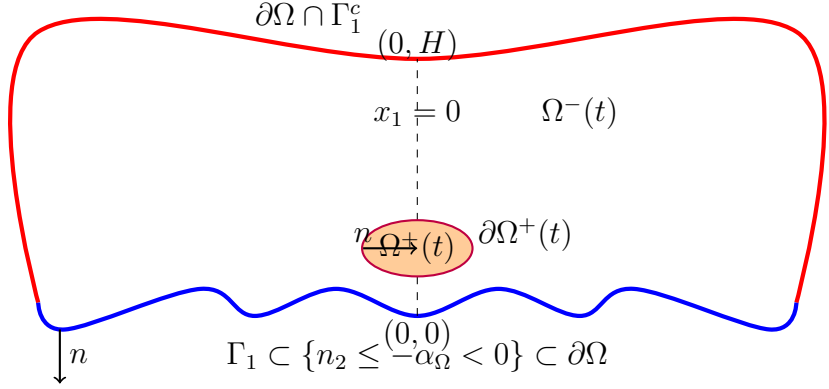
\begin{figure}[h]

\begin{tikzpicture}[scale=0.36]

    \draw[color=blue,ultra thick] plot[smooth,tension=.6] coordinates{( 0,0)(1,-1) (6,0.5) (8,-0.5)(11,0.5) (14, -0.5) (17, 0.5) (20,-0.5)(22,0.5) (27, -1) (28,0) };
    
    \draw[color=red,ultra thick] plot[smooth,tension=.6] coordinates{( 28,0) (28,10) (14,9) (0,10) (0,0)};
\draw[black,dashed] (14,-0.5)-- (14,9);
\draw[purple,ultra thick] (14,2) ellipse (2cm and 1 cm);
\fill[orange!40!white] (14,2) ellipse (2cm and 1 cm);  
\draw (14,2) node { {\large $\Omega^+(t)$}}; 
\draw (18,2.5) node { {\large $\partial\Omega^+(t)$}};
\draw (20,7) node { {\large $\Omega^-(t)$}};
\draw (14,-2) node { {\large $\Gamma_1\subset\{n_2\le -\alpha_\Omega<0\}\subset\partial\Omega$}};
\draw (10,10.5) node { {\large $\partial\Omega\cap\Gamma_1^c$}};
 \draw (14,7) node { {\large $x_1=0$}};   
 \draw[thick,->] (0.8,-1) -- (0.8,-3);
 \draw (1.5,-2) node { {\large $n$}}; 
  \draw[thick,->] (12,2) -- (14,2);
 \draw (12,2.5) node { {\large $n$}};
 \draw (14,-1.2) node { {\large $(0,0)$}}; 
   \draw (14,9.5) node { {\large $(0,H)$}}; 
\end{tikzpicture} 
\caption{ In our convention, $n$ is always exterior to $\Omega^-(t)$ }\label{fig1} 
\end{figure}

We denote by $\Omega$ a domain of class $H^{\f{9}{2}}$ locally on one side of its boundary, which is symmetric with respect to the vertical axis $x_1=0$ and whose boundary $\p\Omega$ is connected. This domain is of height $H>0$ along the vertical axis $x_1=0$, with the bottom point on the vertical axis being $(0,0)$. We also assume that $\Omega$ has a part of its boundary $\Gamma_1$, centered at the origin, with $-L\le x_1\le L$ ($L>0$), under the form of a graph $x_2=f(x_1)$ and thus satisfying
\begin{equation}
\label{graph}
n_2<-\alpha_\Omega<0\,,\ \text{on}\ \Gamma_1\,.
\end{equation}

We then choose $\Omega^+$ such that $\overline{\Omega^+}\subset\Omega$ to be an equally symmetric domain with respect to the vertical axis $x_1=0$, which is of the same regularity class as $\Omega$ and is locally on one side of its boundary. 

We then define the initial fluid domain $\Omega^-=\Omega\cap(\overline{\Omega^+})^c$.

We choose $u^+(0)$ and $u^-(0)$ smooth divergence free velocity fields such that their horizontal component is odd whereas their vertical one is even and satisfying at time $0$, (\ref{ff1.f}) and (\ref{ff1.g}). At time $0$, the center of gravity of $\Omega^+$ is located at $x^+(0)=(0,h)$. Given the symmetry of the initial data with respect to the $x_1=0$ axis, we have that for all time of existence $u^\pm_1(-x_1,x_2,\cdot)=-u^\pm_1(x_1,x_2,\cdot)$,  $u^\pm_2(-x_1,x_2,\cdot)=u^\pm_2(x_1,x_2,\cdot)$.
This implies for the center of gravity of $\ostp$ that $x_1^+(t)=0=v_1^+(t)$.
 This can be seen by setting the fixed-point approach of \cite{ChCoSh2008} in a symmetric setting.
 
 We moreover assume
\begin{equation}
\label{3112.3}
\f{3}{2}|\p\Omega^+|+2 h+2\f{E(0)}{(\rho^+-\rho^-)g|\Omega^+|} +2\f{E(0)}{\sigma} <\min(\f{L}{2},\f{H}{4})\,,
\end{equation}
and 
\begin{equation}
\label{3112.8}
\{(x_1,x_2);\ x_1\in[-\f{L}{2},\f{L}{2}]; f(x_1)\le x_2\le\f{H}{2}\}\subset\Omega\,.
\end{equation}
We also assume:
\begin{equation}
\label{3112.12}
\Gamma_1^c\cap\p\Omega\subset(\p\Omega\cap\{|x_1|\ge L\})\cup (\p\Omega\cap\{x_2\ge \f{H}{2}\})\,,
\end{equation}
where $\Gamma_1$ was defined earlier in this Section.
The first condition can be satisfied by taking the dimensions of the container domain $\Omega$ large relative to $\Omega^+$ and the initial $x^+(0)$, and small square integrable velocities, whereas the second and third ones are conditions on the shape of $\p\Omega$ (if $\Omega$ is for instance of essentially rectangular shape, with four smoothed corners, all these conditions are satisfied).

The conservation of (\ref{ff5}) states:
\begin{equation}
\label{chf1bis}
\sum_\pm\f{\rho^\pm}{2}\int_{\Omega^\pm(t)} |u^\pm|^2\ dx+ (\rho^+-\rho^-)  g x^+_2(t)|\Omega^+|+\sigma |\p\ostp|= E(0)\,. 
\end{equation}
First, (\ref{chf1bis}) shows that 
\begin{equation}
\label{chf3bis}
x^+_2(t)\le \f{E(0)}{(\rho^+-\rho^-)g|\Omega^+|}< \f{H}{8}\,,
\end{equation}
where we used (\ref{3112.3}) to obtain the second inequality. This shows that the center of gravity of $\Omega^+(t)$ stays away from the top of $\partial\Omega$. 
Also (\ref{chf1bis}) shows that
\begin{equation}
\label{0201.1}
\sum_\pm\f{\rho^\pm}{2}\int_{\Omega^\pm(t)} |u^\pm|^2\ dx\le E(0)\,.
\end{equation}

Now, we prove $\ostp$  stays away from the top of $\p\Omega$ and from the lateral sides of $\p\Omega$. Using again (\ref{chf1bis}), we have since $x^+_2(t)\ge 0$ that
\begin{equation}
\label{chf4bis}
|\p\ostp|\le \f{E(0)}{\sigma}\le \min(\f{H}{8},\f{L}{4})\,,
\end{equation}
by using our assumption (\ref{3112.3}). 
Now, let $x^l(t)$ be a point of lowest altitude of $\ostp$ and $x^h(t)$ be a point of highest altitude of $\ostp$. Then, since the straight line from these two points is shorter than any of the two paths along $\p\ostp$ between them, we have:
\begin{equation}
\label{chf6bis}
x^h_2(t)-x^l_2(t)\le |x^h(t)-x^l(t)|\le \f{|\p\ostp|}{2}\le \f{H}{16}\,,
\end{equation}
where we used (\ref{chf4bis}). Thus, 
$$x^h_2(t)\le \f{H}{16}+x^l_2(t)\le \f{H}{16}+x^+_2(t)\,,$$ which with (\ref{chf3bis}) provides
\begin{equation}
\label{chf7bis}
x^h_2(t)\le \f{H}{16}+\f{H}{8}<\f{H}{4}\,.
\end{equation}

By introducing a most on the left point $x^L(t)$ and a most on the right point $x^R(t)$ of $\ostp$, we have similarly:
$$2x^R_1(t)=x^R_1(t)-x^L_1(t)\le \f{|\p\ostp|}{2}\le \f{L}{8}\,,$$ where we used (\ref{chf4bis}). Thus,
\begin{subequations}
\begin{align}
0\le & x^R_1(t)\le \f{L}{16}\,,\label{2812.7a}\\
-\f{L}{16}\le & x^L_1(t)\le 0\,,\label{2812.7b}\,,
\end{align}
\end{subequations}
Propositions (\ref{chf7bis}), (\ref{2812.7a}), (\ref{2812.7b}) then show that for all time of existence
\begin{equation}
\label{3112.9}
\ostp\subset \{(x_1,x_2);\ x_1\in [-\f{L}{16},\f{L}{16}]; f(x_1)\le x_2\le \f{H}{4}\}\subset\Omega\,,
\end{equation}
with our assumption (\ref{3112.8}). Therefore, due to our assumption (\ref{3112.12}), for all time of existence
\begin{equation}
\label{3112.10}
d(\Omega^+(t),\p\Omega\cap \Gamma_1^c)\ge D=\min(\f{H}{4},\f{15L}{16})>0\,,
\end{equation}
where $\Gamma_1\subset\p\Omega$ was defined earlier in this Section as the bottom part of $\p\Omega$ under the form of a graph.

\section{Evolution of the center of gravity of the moving fluid bubble $\ostp$ }
\label{section7}

We have the fundamental equation for the center of mass:
\begin{align}
m_+\ \frac{dv^+}{dt}(t)=&\rho^+ \int_{\Omega^+} \f{dv^+}{dt}(x,t)\ dx\n\\
=&\rho^+ \int_{\Omega^+(t)} u^+_t+u^+\cdot\nabla u^+\ dx\n\\
=&-\int_{\Omega^+(t)} \nabla p^++\rho^+ g e_2\ dx\n\\
=&\int_{\p\Omega^+(t)} p^+\ n\ dl(t)-m_+ g e_2
\label{ef1}\,,
\end{align}
where we remind $n$ is the outer unit normal to $\oftp$, pointing inside $\ostp$, which explains the sign in the boundary integral in (\ref{ef1}). 
Using our boundary condition (\ref{ff1.e}), this provides:
\begin{align}
m_+\ \frac{dv^+}{dt}
=&\int_{\p\ostp}(p^-\ n+\sigma \nabla_{\tau}(\tau)) \ dl(t)-m_+ ge_2\n\\
=&\int_{\p\ostp}p^-\ n \ dl(t)-m_+ ge_2\,.
\label{ef2}
\end{align}
where we used
$$\int_{\p\ostp}\nabla_\tau(\tau)\ dl=0\,,$$ for any closed smooth curve such as $\p\ostp$ (so long as the smooth solution exists). 

By integrating by parts in $\oftp$:
\begin{equation*}
\int_{\oftp}\nabla p^-\ dx=\int_{\p\ostp}p^-\ n\ dl(t)+\int_{\p\Omega} p^-\ n\ dl\,.
\end{equation*}
This provides by substitution in (\ref{ef2}):
\begin{equation*}
m_+\ \frac{dv^+}{dt}=-\int_{\p\Omega}p^-\ n\ dl+\int_{\oftp} \nabla p^-\ dx -m_+ ge_2\,,
\end{equation*}
which with the Euler equations provides:
\begin{align}
m_+\ \frac{dv^+}{dt}=&-\int_{\p\Omega}p^-\ n\ dl-\rho^- \int_{\oftp} u^-_t+u^-\cdot\nabla u^-+ge_2\ dx -m_+ ge_2\,,\n\\
=&-\int_{\p\Omega}p^-\ n\ dl-\rho^- \frac{d}{dt}\int_{\oftp} u^-\ dx -(m_++\rho^- |\Omega^-|) ge_2\,.\label{e2}
\end{align}
Next, we notice that on $\p\Omega$, thanks to (\ref{e3bis}) we have:
\begin{equation}
\label{e7}
\rho^- u^-_t\cdot\tau +\nabla_\tau( p^-+\rho^-\f{|u^-|^2}{2})=-\rho^- ge_2\cdot\tau+\omega^- \underbrace{ u^-\cdot n}_{=0\ \ \text{on}\ \p\Omega}=-\rho^- ge_2\cdot\tau\,.
\end{equation}
We now denote by $\theta:[0,1]\rightarrow\p\Omega$ a $1$-periodic smooth parameterization of $\p\Omega$ with $\theta(0)=(0,H)$. We integrate (\ref{e7}) along $\p\Omega$ between $\theta(0)$ and $\theta(s)$:
\begin{equation*}
[(p^-+\rho^-\f{|u^-|^2}{2})(\theta(\cdot),t)]_0^s=-\int_0^s (\rho^- ge_2\cdot\tau+\rho^- u^-_t\cdot\tau)(\theta(\alpha),t)\underbrace{|\theta'(\alpha)| d\alpha}_{=dl}\,,
\end{equation*}
which implies by integrating (in the $s$ variable) along $\p\Omega$
\begin{align*}
-\int_{\p\Omega}p^-\ n\ dl=& \rho^-\int_{\p\Omega}\f{|u^-|^2}{2}\ n\ dl-(p^-+\rho^-\f{|u^-|^2}{2})(\theta(0),t) \int_{\p\Omega}  n\ dl\\
&+ \int_0^1 \int_0^s \rho^- (ge_2+ u^-_t)\cdot\tau(\theta(\alpha),t)|\theta'(\alpha)| d\alpha\ n(\theta(s))\ {|\theta'(s)|\ ds}\,.
\end{align*}
Since $\p\Omega$ is a closed curve, $\int_{\p\Omega} n\ dl=0$, and thus the previous relation becomes:
\begin{align}
-\int_{\p\Omega}p^-\ n\ dl=& \rho^-\int_{\p\Omega}\f{|u^-|^2}{2}\ n\ dl\n\\
&+ \int_0^1 \int_0^s \rho^- (ge_2+ u^-_t)\cdot\tau(\theta(\alpha),t)|\theta'(\alpha)| d\alpha\ n(\theta(s))\ |\theta'(s)|\ ds\,.\label{e8}
\end{align}
We now substitute (\ref{e8}) into (\ref{e2}), leading to:
\begin{align}
m_+\ \frac{dv^+}{dt}
=&-\rho^- \frac{d}{dt}\int_{\oftp} u^-\ dx -(m_++\rho^- |\Omega^-|) ge_2+\rho^-\int_{\p\Omega}\f{|u^-|^2}{2}\ n\ dl\n\\
&+ \int_0^1 \int_0^s (\rho^- ge_2\cdot\tau+\rho^- u^-_t\cdot\tau)(\theta(\alpha),t)|\theta'(\alpha)| d\alpha\ n(\theta(s))\ |\theta'(s)|\ ds\,.\label{e9}
\end{align}
We now write in a much simpler way the fourth term on the right-hand side of this equation. In order to do so, we define $f(x)=x_2$, so that $\nabla f=e_2$ and $\nabla_\tau f=e_2\cdot\tau$. Therefore,
\begin{equation}
\label{e10}
f(\theta(s))=f(\theta(0))+\int_0^s \underbrace{e_2\cdot\tau(\theta(\alpha))}_{\nabla_\tau f (\theta(\alpha))}\underbrace{|\theta'(\alpha)|\ d\alpha}_{dl}\,.
\end{equation}
Next, since
\begin{equation}
\label{e11}
\int_\Omega e_2\ dx=\int_\Omega \nabla f\ dx=\int_{\p\Omega} f\ n\ dl\,,
\end{equation}
substituting (\ref{e10}) in (\ref{e11}) provides (using $f(\theta(0))\int_{\p\Omega} n\ dl=0$):
\begin{equation}
\label{e12}
\int_\Omega e_2\ dx=\int_0^1 \int_0^s e_2\cdot\tau (\theta(\alpha))|\theta'(\alpha)| d\alpha\ n(\theta(s))\ |\theta'(s)|\ ds\,.
\end{equation}
Using (\ref{e12}) in (\ref{e9}) then yields:
\begin{align}
m_+\ \frac{dv^+}{dt}
=&\rho^-\int_{\p\Omega}\f{|u^-|^2}{2}\ n\ dl-\rho^- \frac{d}{dt}\int_{\oftp} u^-\ dx -(m_++\rho^- |\Omega^-|-\rho^-|\Omega|) ge_2\n\\
&+ \int_0^1 \int_0^s \rho^- u^-_t\cdot\tau(\theta(\alpha),t)|\theta'(\alpha)| d\alpha\ n(\theta(s))\ |\theta'(s)|\ ds\n\\
=&\rho^-\int_{\p\Omega}\f{|u^-|^2}{2}\ n\ dl-\rho^- \frac{d}{dt}\int_{\oftp} u^-\ dx -(m_+-\rho^-|\Omega^+|) ge_2\n\\
&+ \f{d}{dt}\int_0^1 \int_0^s \rho^- u^-\cdot\tau(\theta(\alpha),t)|\theta'(\alpha)| d\alpha\ n(\theta(s))\ |\theta'(s)|\ ds
\,.\label{e13}
\end{align}
Defining
\begin{equation}
\label{e14}
F(t)=\int_0^1 \int_0^s \rho^- u^-\cdot\tau(\theta(\alpha))|\theta'(\alpha)| d\alpha\ n(\theta(s))\ |\theta'(s)|\ ds\,,
\end{equation}
we have by integrating(\ref{e13}) in time
\begin{align}
F_2(t)=& m_s v_2^+(t)+\rho^-\int_{\oftp} u^-_2 dx -\rho^- \int_0^t \int_{\p\Omega}\f{|u^-|^2}{2}n_2 dl dt+\underbrace{(\rho^+-\rho^-)}_{> 0}|\Omega^+| gt\n\\
& +\underbrace{F_2(0)-m_+ v_2^+(0)-\rho^-\int_{\Omega^-} u^-_2(\cdot,0)\ dx}_{C_0}\,.
\label{e15}
\end{align} 

\begin{remark}
We notice the computations leading to (\ref{e15}) came purely from using the incompressible Euler equations with gravity in $\oftp$ in the relation (\ref{ef2}), and are valid for any law governing the phase $\Omega^+(t)$, including the case of the rigid body considered later in this paper.
\end{remark}
 
We next rewrite $F_2$ in a simpler way. 

First, since $n_2=\tau_1$, we have
\begin{equation*}
F_2(t)=\int_0^1 \int_0^s \rho^- u^-\cdot\tau(\theta(\alpha))|\theta'(\alpha)| d\alpha\ \tau_1(\theta(s))\ |\theta'(s)|\ ds\,,
\end{equation*}
which by integration by parts provides:
\begin{equation}
\label{28.01}
F_2(t)=-\int_0^1 \rho^- u^-\cdot\tau(\theta(s))|\theta'(s)| \ \int_0^s \tau_1(\theta(\alpha))\ |\theta'(\alpha)|d\alpha\ ds\,.
\end{equation}
Note here that we used the fact that $\int_{\p\Omega} \tau_1\ dl=0$.

Moreover, in the same way as we obtained (\ref{e10}), this time for 
 $f(x)=x_1$, so that $\nabla f=e_1$ and therefore $\nabla_\tau f=e_1\cdot\tau$, we have
\begin{equation}
\label{28.02}
f(\theta(s))=f(\theta(0))+\int_0^s \underbrace{e_1\cdot\tau(\theta(\alpha))}_{\nabla_\tau f (\theta(\alpha))}\underbrace{|\theta'(\alpha)|\ d\alpha}_{dl}\,,
\end{equation}
which by substitution in (\ref{28.01}) provides
\begin{equation*}
F_2(t)=-\int_0^1 \rho^- u^-\cdot\tau(\theta(s))|\theta'(s)| \ (\theta_1(s)-\theta_1(0))\ ds\,.
\end{equation*}
Using  $\theta_1(0)=0$, this yields
\begin{align}
F_2(t)=&-\int_0^1 \rho^- u^-\cdot\tau(\theta(s))|\theta'(s)| \ \theta_1(s)\ ds\n\\
=&-\int_{\p\Omega}\rho^- x_1\ u^-\cdot\tau\ dl\label{28.03}\,.
\end{align}
Substituting (\ref{28.03}) in (\ref{e13}) provides
\begin{align}
m_s\ \frac{dv^+_2}{dt}
=&\rho^-\int_{\p\Omega}\f{|u^-|^2}{2}\ n_2\ dl-\rho^- \frac{d}{dt}\int_{\oftp} u^-_2\ dx -(m_+-\rho^-|\Omega^+|) g\n\\
&- \f{d}{dt}\int_{\p\Omega}\rho^- x_1\ u^-\cdot\tau\ dl
\,.\label{28.10}
\end{align} 

\section{Finite-time singularity formation for the vortex sheet problem with surface tension}
\label{section8}

We note that from our energy conservation (\ref{chf1bis}) and (\ref{ff9}), the first and second terms on the right-hand side of (\ref{e15}) are controlled for all time of existence by a constant independent of time, while the fourth term is linear in time. We now address the question of the third term, which is not sign definite across $\p\Omega$, due to the presence of $n_2$. 

We remind that our assumptions from Section \ref{section6} imply that we can split $\p\Omega$ into the graph $\Gamma_1$, centered on the vertical axis $x_1=0$, below the (potentially) falling moving body in the fluid, and where $n_2\le -\alpha_\Omega<0$ and its complementary, where we will show the integral is small relative to the fourth term of (\ref{e15}).

From (\ref{e15}) we infer:
\begin{align}
F_2(t)\ge & m_+ v_2^+(t)+\rho^-\int_{\oftp} u^-_2\ dx -\rho^- \int_0^t \int_{\Gamma_1^c\cap\p\Omega}\f{|u^-|^2}{2}n_2\ dl\ dt\n\\
&+{\rho^-}\alpha_\Omega\int_0^t \int_{\Gamma_1}\f{|u^-|^2}{2}\ dl\ dt+\underbrace{(\rho^+-\rho^-)}_{> 0}|\Omega^+| gt+C_0\n\\
\ge & m_+ v_2^+(t)+\rho^-\int_{\oftp} u^-_2\ dx \ dx +C_0\n\\
&-\rho^- \int_0^t \int_{\Gamma_1^c\cap\p\Omega}\f{|u^-|^2}{2}n_2\ dl\ dt+\f{1}{4}{(\rho^+-\rho^-)}|\Omega^+| gt\n\\
&+{\rho^-}\alpha_\Omega\int_0^t \int_{\Gamma_1}\f{|u^-|^2}{2}\ dl\ dt+\f{3}{4}{(\rho^+-\rho^-)}|\Omega^+| gt
,.\label{e16}
\end{align}

We will prove later on that for initial height $h$ and initial velocities satisfying (\ref{3112.1}) stated later,
\begin{equation}
\label{e17}
(1+\alpha_\Omega) \left|\rho^- \int_0^t \int_{\Gamma_1^c\cap\p\Omega}{|u^-|^2}\ dl\ dt\right|\le \f{1}{4}(\rho^+-\rho^-)|\Omega^+| gt\,.
\end{equation}
Using this property (\ref{e17}), we have by $\p\Omega=\Gamma_1\cup(\Gamma_1^c\cap\p\Omega)$:

\begin{equation*}
 {\rho^-}\alpha_\Omega\int_0^t \int_{\Gamma_1}\f{|u^-|^2}{2}\ dl\ dt+\f{1}{8}(\rho^+-\rho^-)|\Omega^+| gt \ge {\rho^-}\alpha_\Omega\int_0^t \int_{\p\Omega}\f{|u^-|^2}{2}\ dl\ dt \,.
\end{equation*}
and thus
\begin{equation}
\label{e18}
 {\rho^-}\alpha_\Omega\int_0^t \int_{\Gamma_1}\f{|u^-|^2}{2}\ dl\ dt+\f{1}{4}(\rho^+-\rho^-)|\Omega^+| gt \ge {\rho^-}\alpha_\Omega\int_0^t \int_{\p\Omega}\f{|u^-|^2}{2}\ dl\ dt \,.
\end{equation}
Using again (\ref{e17}), we also have
\begin{equation}
\label{e19}
-\rho^- \int_0^t \int_{\Gamma_1^c\cap\p\Omega}\f{|u^-|^2}{2}n_2\ dl\ dt + \f{1}{4}(\rho^+-\rho^-)|\Omega^+| gt\ge 0\,.
\end{equation}
Using (\ref{e18}) and (\ref{e19}) in (\ref{e16}), we infer that
\begin{align}
F_2(t)\ge & m_+ v_2^+(t)+\rho^-\int_{\oftp} u^-_2 dx 
+{\rho^-}\alpha_\Omega \int_0^t \int_{\p\Omega}\f{|u^-|^2}{2} dl dt\n\\
&+\f{\rho^+-\rho^-}{2}|\Omega^+|gt +{C_0}\,.\label{e20}
\end{align}
On the other hand, given (\ref{28.03}) for $F$, we have the existence of $\tilde C_{\Omega}>0$ (depending on $\Omega$) such that
\begin{equation}
\label{e21}
|F(t)|\le \tilde C_\Omega \rho^- \int_{\p\Omega} |u^-|(\cdot,t)\ dl\,.
\end{equation}
Using (\ref{e21}) in (\ref{e20}), we obtain:
\begin{align}
 \tilde C_\Omega \rho^- \int_{\p\Omega} |u^-|(\cdot,t) dl \ge & m_+ v_2^+(t)+\rho^-\int_{\oftp} u^-_2\ dx  +{\rho^-} \alpha_\Omega\int_0^t \int_{\p\Omega}\f{|u^-|^2}{2}\ dl\ dt\n\\
 &+\f{\rho^+-\rho^-}{2}|\Omega^+| gt- |C_0| \,.\label{e22}
\end{align}
By Cauchy-Schwarz, this implies
\begin{align}
\tilde C_\Omega \rho^- \int_{\p\Omega} |u^-|(\cdot,t)\ dl \ge & m_+ v_2^+(t)+\rho^-\int_{\oftp} u^-_2\ dx - |C_0|\n\\
& +\f{\rho^-\alpha_\Omega}{2t|\p\Omega|} \left(\int_0^t \int_{\p\Omega}{|u^-|}\ dl dt\right)^2 +\f{1}{2}(\rho^+-\rho^-)|\Omega^+| gt\,.\label{e23}
\end{align}
Using our energy bounds (\ref{0201.1}) and (\ref{ff9}) in (\ref{e23}), we have:
\begin{align}
 \tilde C_\Omega \rho^- \int_{\p\Omega} |u^-|(\cdot,t)\ dl \ge & - m_+ \sqrt{\f{2E(0)}{m_+}}- \rho^- \sqrt{\f{2E(0)}{\rho^-}}\sqrt{|\Omega^-|}- |C_0|\n\\
& +\f{\rho^-\alpha_\Omega}{2t|\p\Omega|} \left(\int_0^t \int_{\p\Omega}{|u^-|}\ dl\ dt\right)^2 +\f{1}{2}(\rho^+-\rho^-)|\Omega^+| gt\,.\label{e23bis}
\end{align}

Let 
\begin{equation}
\label{e24}
f(t)=\int_0^t \int_{\p\Omega}{|u^-|}\ dl\ dt\,.
\end{equation}
From (\ref{e23bis}), we have that with
\begin{equation}
\label{e24bis}
t_0=\f{4}{g|\Omega^+|(\rho^+-\rho^-)} \left(  \sqrt{{2E(0)}{m_+}}+  \sqrt{2E(0)\rho^-|\Omega^-|}+  |C_0|\right)\,,
\end{equation} 
for all $t\ge t_0$, (\ref{e23bis}) implies:
\begin{equation*}
\tilde C_\Omega\rho^- f'(t)\ge \f{\rho^-\alpha_\Omega}{2t|\p\Omega|}f^2+ \f{1}{4}(\rho^+-\rho^-)|\Omega^+| gt>0\,.
\end{equation*}
Therefore, for all $t\ge t_0$, $f(t)>0$ and 
\begin{equation}
\label{e25}
\f{f'}{f^2} \ge \f{\alpha_\Omega}{2|\p\Omega|\tilde C_\Omega t}\,,
\end{equation}
which by integration from $t_0$ to $t\ge t_0$ provides:
$$-\f{1}{f(t)}+\f{1}{f(t_0)}\ge \f{\alpha_\Omega}{2|\p\Omega|\tilde C_\Omega }\ln(\f{t}{t_0})\,.
$$
Therefore,
$$0< \f{1}{f(t)}\le \f{1}{f(t_0)}- \f{\alpha_\Omega}{2|\p\Omega|\tilde C_\Omega }\ln(\f{t}{t_0})\,,
$$
which shows that for $t\ge t_0\ e^{\f{2|\p\Omega|\tilde C_\Omega}{\alpha_\Omega f(t_0)}}$, we have
$$0<\f{1}{f(t)}\le 0\,,$$ which is an obvious impossibility. Therefore, the maximal time of existence of a smooth solution $T_{max}>0$ satisfies
\begin{equation}
\label{e26}
T_{max}\le t_0\ e^{\f{2|\p\Omega|\tilde C_\Omega}{\alpha_\Omega f(t_0)}}\,.
\end{equation}

We now have to turn back to proving our missing estimate (\ref{e17}) (provided our initial data satisfy (\ref{3112.1})), which controls the velocity on $\Gamma_1^c$ (that we are sure the moving bubble $\ostp$ stays away from, given (\ref{3112.10})). Here the difficulty is to get the precise bound given by (\ref{e17}) and not just a generic constant, or a constant greater than the majorant of (\ref{e17}), and it calls for subtle observations of elliptic and geometric natures.

Our starting point is the fact that $u^-$ being divergence free
\begin{equation}
\label{e30}
u^-=\nabla^\perp\phi\,,\text{in}\ \Omega^+(t)\,,
\end{equation}
with 
\begin{equation}
\label{e31}
\phi=0\,,\text{on}\ \p\Omega\,,
\end{equation}
(we will not need the condition $\nabla^\perp\phi\cdot n=u^+\cdot n$ on $\p\Omega^+(t)$)
and
\begin{equation}
\label{e32}
\Delta\phi=\omega^-\,,\text{in}\ \Omega^f(t)\,.
\end{equation}
We will also need the fact that 
\begin{equation}
\|\nabla\phi\|^2_{L^2(\oftp)}=\|u^f\|^2_{L^2(\oftp)}\le 2 \f{E(0)}{\rho^-} \,,
\end{equation} 
due to(\ref{0201.1}). 
\begin{figure}[h]

\begin{tikzpicture}[scale=0.36]

    \draw[color=blue,ultra thick] plot[smooth,tension=.6] coordinates{( 0,0)(1,-1) (6,0.5) (8,-0.5)(11,0.5) (14, -0.5) (17, 0.5) (20,-0.5)(22,0.5) (27, -1) (28,0) };
    
    \draw[color=red,ultra thick] plot[smooth,tension=.6] coordinates{( 28,0) (28,10) (14,9) (0,10) (0,0)};
\draw[black,dashed] (14,-0.5)-- (14,9);
\draw[purple,ultra thick] (14,1) ellipse (1cm and 0.5 cm);
\fill[orange!40!white] (14,1) ellipse (1cm and 0.5 cm);  
\draw (14,1) node { { $\Omega^+(t)$}}; 
\draw (16,1.5) node { { $\partial\Omega^+(t)$}};
\draw (20,7) node { {\large $\Omega_2^\alpha\subset\Omega^-(t)$}};
\draw (2,5.6) node { { \large $\Gamma_2^\alpha$}}; 
\draw (14,-2) node { {\large $\Gamma_1\subset\{n_2\le -\alpha_\Omega<0\}\subset\partial\Omega$}};
\draw (10,10.5) node { {\large $\partial\Omega\cap\Gamma_1^c$}};
 \draw (14,7) node { {\large $x_1=0$}};   
 \draw[thick,->] (0.8,-1) -- (0.8,-3);
 \draw (1.5,-2) node { {\large $n$}}; 
  \draw (14,-1.2) node { {\large $(0,0)$}}; 
   \draw (14,9.5) node { {\large $(0,H)$}}; 
    \draw[color=purple, ultra thick] (-1,5)--(29,5); 
\end{tikzpicture} 
\caption{ $\Omega^+(t)$ stays below $\Gamma_2^\alpha$ for $\alpha>\f{H}{4}$}\label{fig1.bis} 
\end{figure}
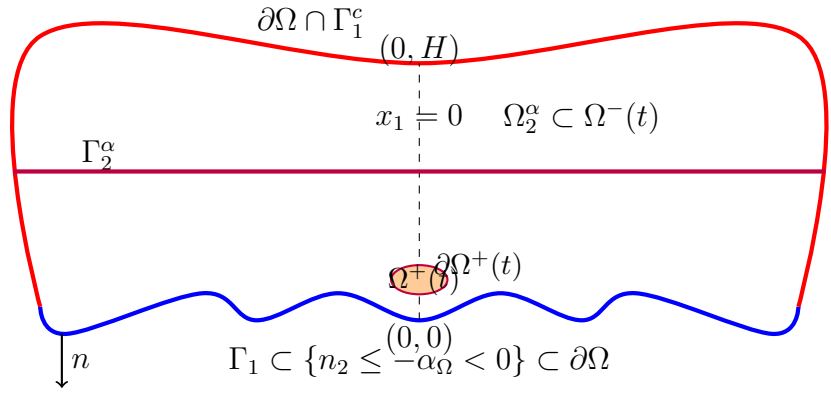
We now define for any $\alpha>\f{H}{4}$
\begin{equation}
\label{e33}
\Omega_{2}^{\alpha}=\{x\in\Omega;\ x_2\ge \alpha\}\,,
\end{equation}
and 
\begin{equation}
\label{e34}
\Gamma_{2}^{\alpha}=\{x\in\Omega;\ x_2= \alpha\}\,.
\end{equation}

From our relation (\ref{3112.9}), we know that for $\alpha> \f{H}{4}$, $\Omega_{2}^{\alpha}$ does not intersect $\Omega^+(t)$ for all time of existence, and we will work with such values of $\alpha$ in what follows.

Taking $\tilde n$ as being our global vector field extending $n$ into $\Omega$ defined in Section \ref{section2}, we now take $|\tilde n|^2\nabla_{\tilde n}\phi$ as test function for (\ref{e32}) and integrate the relation in $\Omega_{2}^{\alpha}$.

We first notice that for any unit vector $a$, if $b=a^\perp$, we have:
$$\omega^-=\Delta\phi=a_i a_j \f{\p^2\phi}{\p x_i\p x_j} +b_i b_j \f{\p^2\phi}{\p x_i\p x_j}\,.$$
Therefore, for any vector $a$, if $b=a^\perp$,
$$|a|^2\omega^-=|a|^2 \Delta\phi=a_i a_j \f{\p^2\phi}{\p x_i\p x_j} +b_i b_j \f{\p^2\phi}{\p x_i\p x_j}\,.$$
We now simply use this expansion of $\Delta$ in the orthogonal $(\tilde\tau(x),\tilde n(x))$ basis at each point $x\in\oftp$, to get
\begin{equation}
\label{2812.8}
|\tilde n(x)|^2\omega^-=|\tilde n(x)|^2\Delta\phi= \ttau_i(x)\ttau_j(x) \f{\p^2\phi}{\p x_i\p x_j}(x,t) +\tn_i(x) \tn_j(x) \f{\p^2\phi}{\p x_i\p x_j}(x,t)\,.
\end{equation}
Integration by parts in $\Omega_{2}^{\alpha}\subset\oftp$ (and remembering that the normal exterior vector to $\Gamma_{2}^{\alpha}$ is $-e_2$ and to $\p\Omega$ is $n=\tilde n$) provides us with
\begin{align*}
\int_{\Omega_{2}^{\alpha}} |\tilde n|^2\omega^-\nabla_{\tn}\phi dx&=\int_{\Omega_{2}^{\alpha}} (\ttau_i\ttau_j \f{\p^2\phi}{\p x_i\p x_j} +\tn_i \tn_j \f{\p^2\phi}{\p x_i\p x_j})\tn_k \f{\p\phi}{\p x_k}\ dx\\
&=-\underbrace{\int_{\Omega_{2}^{\alpha}}  \f{\p\phi}{\p x_j} \f{\p(\ttau_i\ttau_j \tn_k \f{\p\phi}{\p x_k})}{\p x_i} dx}_{I_1}- \underbrace{\int_{\Omega_{2}^{\alpha}} \f{\p\phi}{\p x_j} \f{\p (\tn_i \tn_j \tn_k \f{\p\phi}{\p x_k})}{\p x_i} dx}_{I_2}\\
&-\int_{\Gamma_{2}^{\alpha}} \f{\p\phi}{\p x_j} \ttau_2 \ttau_j \tn_k \f{\p\phi}{\p x_k}  + \tn_2 \tn_j \tn_k \f{\p\phi}{\p x_j} \f{\p\phi}{\p x_k}\ dx_1\\
&+\int_{\p\Omega_{2}^{\alpha}\cap\p\Omega} \f{\p\phi}{\p x_j} \underbrace{\ttau_i \tn_i}_{=0} \ttau_j \tn_k \f{\p\phi}{\p x_k} + \f{\p\phi}{\p x_j} \underbrace{\tn_i n_i}_{=1\ \text{on}\ \p\Omega} \tn_j {\tn_k \f{\p\phi}{\p x_k}}dl\,.
\end{align*}
Thus,
\begin{align}
I_1+I_2=&-\int_{\Gamma_{2}^{\alpha}} \f{\p\phi}{\p x_j} \ttau_2 \ttau_j \tn_k \f{\p\phi}{\p x_k}  + \tn_2 \tn_j \tn_k \f{\p\phi}{\p x_j}\f{\p\phi}{\p x_k}\ dx_1\n\\
&+\int_{\p\Omega_{2}^{\alpha}\cap\p\Omega} |\nabla_n\phi|^2\ dl-\int_{\Omega_{2}^{\alpha}} |\tilde n|^2\omega^-\nabla_{\tn}\phi\ dx\,.
\label{e35}
\end{align}
We next rewrite $I_1$ and $I_2$. From the definition,
\begin{align}
I_2=&\int_{\Omega_{2}^{\alpha}} \tn_i\tn_j\f{\p\phi}{\p x_j} \f{\p ( \tn_k \f{\p\phi}{\p x_k})}{\p x_i}\ dx+\int_{\Omega_{2}^{\alpha}} \f{\p(\tn_i\tn_j)}{\p x_i} \f{\p\phi}{\p x_j}  \tn_k \f{\p\phi}{\p x_k}\ dx\n\\
=&\f{1}{2} \int_{\Omega_{2}^{\alpha}} \tn_i\f{\p |\nabla_{\tn}\phi|^2 }{\p x_i}\ dx+\int_{\Omega_{2}^{\alpha}} \f{\p(\tn_i\tn_j)}{\p x_i} \f{\p\phi}{\p x_j}  \tn_k \f{\p\phi}{\p x_k}\ dx\n\\
=&- \int_{\Omega_{2}^{\alpha}} \f{\p \tn_i}{\p x_i} \f{|\nabla_{\tn}\phi|^2}{2} dx+
\int_{\p\Omega_{2}^{\alpha}\cap \p\Omega} \underbrace{|\tn|^2}_{=1\ \text{on}\ \p\Omega} \f{|\nabla_{\tn}\phi|^2}{2} dl-\int_{\Gamma_{2}^{\alpha}} \tn_2\f{ |\nabla_{\tn}\phi|^2 }{2} dx_1\n\\
&+\int_{\Omega_{2}^{\alpha}} \f{\p(\tn_i\tn_j)}{\p x_i} \f{\p\phi}{\p x_j}  \tn_k \f{\p\phi}{\p x_k}\ dx\,.\label{e36}
\end{align}
We next move to $I_1$:
\begin{align}
I_1=&\int_{\Omega_{2}^{\alpha}}  \ttau_i\ttau_j\tn_k \f{\p\phi}{\p x_j} \f{\p^2\phi}{\p x_k\p x_i}\ dx+ \int_{\Omega_{2}^{\alpha}}  \f{\p(\ttau_i\ttau_j\tn_k)}{\p x_i} \f{\p\phi}{\p x_j} \f{\p\phi}{\p x_k}\ dx\n\\
=&\int_{\Omega_{2}^{\alpha}}  \ttau_j\f{\p\phi}{\p x_j} \tn_k\f{\p(\ttau_i \f{\p\phi}{\p x_i})}{\p x_k} dx-\int_{\Omega_{2}^{\alpha}}  \ttau_j\f{\p\phi}{\p x_j} \tn_k\f{\p\ttau_i}{\p x_k}\f{\p\phi}{\p x_i} dx\n\\
&+ \int_{\Omega_{2}^{\alpha}}  \f{\p(\ttau_i\ttau_j\tn_k)}{\p x_i} \f{\p\phi}{\p x_j} \f{\p\phi}{\p x_k} dx\n\\
=&\int_{\Omega_{2}^{\alpha}}   \f{\tn_k}{2}\f{\p |\nabla_{\ttau}\phi|^2}{\p x_k} dx-\int_{\Omega_{2}^{\alpha}}  \ttau_j\f{\p\phi}{\p x_j} \tn_k\f{\p\ttau_i}{\p x_k}\f{\p\phi}{\p x_i} dx\n\\
&+ \int_{\Omega_{2}^{\alpha}}  \f{\p(\ttau_i\ttau_j\tn_k)}{\p x_i} \f{\p\phi}{\p x_j} \f{\p\phi}{\p x_k} dx\n\\
=&-\int_{\Omega_{2}^{\alpha}}   \f{\p \tn_k}{\p x_k}  \f{|\nabla_{\ttau}\phi|^2}{2}\ dx
+\int_{\p\Omega_{2}^{\alpha}\cap\p\Omega} \f{|\tn|^2 }{2} \underbrace{|\nabla_{\ttau}\phi|^2}_{=0\ \text{on}\ \p\Omega}\ dl
-\int_{\Gamma_{2}^{\alpha}}   \f{\tn_2}{2}|\nabla_{\ttau}\phi|^2\ dx_1\n\\
&-\int_{\Omega_{2}^{\alpha}}  \ttau_j\f{\p\phi}{\p x_j} \tn_k\f{\p\ttau_i }{\p x_k}\f{\p\phi}{\p x_i}\ dx+ \int_{\Omega_{2}^{\alpha}}  \f{\p(\ttau_i\ttau_j\tn_k)}{\p x_i} \f{\p\phi}{\p x_j} \f{\p\phi}{\p x_k}\ dx\,.\label{e37}
\end{align}
By gathering (\ref{e35}), (\ref{e36}) and (\ref{e37}) we obtain:
\begin{align}
\int_{\p\Omega_{2}^{\alpha}\cap\p\Omega}  \f{ |\nabla_{n}\phi|^2}{2} dl=&-\f{1}{2}\int_{\Omega_{2}^{\alpha}}   \f{\p \tn_k}{\p x_k}  |\nabla_{\ttau}\phi|^2\ dx-\f{1}{2}\int_{\Gamma_{2}^{\alpha}}   \tn_2|\nabla_{\ttau}\phi|^2\ dx_1\n\\
&-\int_{\Omega_{2}^{\alpha}}  \ttau_j\f{\p\phi}{\p x_j} \tn_k\f{\p\ttau_i }{\p x_k}\f{\p\phi}{\p x_i} dx+ \int_{\Omega_{2}^{\alpha}}  \f{\p(\ttau_i\ttau_j\tn_k)}{\p x_i} \f{\p\phi}{\p x_j} \f{\p\phi}{\p x_k} dx\n\\
&-\f{1}{2} \int_{\Omega_{2}^{\alpha}} \f{\p \tn_i}{\p x_i} |\nabla_{\tn}\phi|^2 -\f{1}{2}\int_{\Gamma_{2}^{\alpha}} \tn_2 |\nabla_{\tn}\phi|^2 \ dx_1\n\\
&+\int_{\Omega_{2}^{\alpha}} \f{\p(\tn_i\tn_j)}{\p x_i} \f{\p\phi}{\p x_j}  \tn_k \f{\p\phi}{\p x_k}\ dx+\int_{\Omega_{2}^{\alpha}} |\tilde n|^2\omega^-\nabla_{\tn}\phi\ dx\n\\
&+\int_{\Gamma_{2}^{\alpha}} \f{\p\phi}{\p x_j} \ttau_2 \ttau_j \tn_k \f{\p\phi}{\p x_k}  +\tn_2 \tn_j \tn_k  \f{\p\phi}{\p x_j}\f{\p\phi}{\p x_k}\ dx_1\,.
\label{e38}
\end{align}
In what follows, $C_i$ is a generic constant which does not depend on our initial velocity and height $h$. Due to (\ref{maximumprinciple}), (\ref{choice2}) and (\ref{0201.1}), we have
that 
\begin{equation}
\label{e39}
\left|\int_{\Omega_{2}^{\alpha}} \f{\p(\tn_i\tn_j)}{\p x_i} \f{\p\phi}{\p x_j}  \tn_k \f{\p\phi}{\p x_k}\ dx\right|\le {C_1}\ \|u^-\|^2_{L^2(\oftp)}\le {C_2} E(0)\,,
\end{equation}
with similar estimates for each of the integrals on $\Omega_{2,\alpha}$ appearing on the right-hand side of (\ref{e38}). With (\ref{e38}) and (\ref{e39}), we then obtain:
\begin{equation*}
\int_{\{x_2\ge\alpha\}\cap\p\Omega}   |\nabla_{\tn}\phi|^2 dl\le {C_3}  (E(0)+\|\omega^-_0\|^2_{L^2(\Omega^-)}) +C_4\int_{\Gamma_{2}^{\alpha}} |u^-|^2 dx_1\,.
\end{equation*}
Remembering
that $u\cdot n=\nabla_{\tau}\phi=0$ on $\p\Omega$, and using (\ref{0201.1}) we deduce that for any $\alpha\ge \f{H}{4}$:
\begin{equation}
\label{3112.11}
\int_{\{x_2\ge\alpha\}\cap\p\Omega}   |u^-|^2\ dl\le {C_3} (\underbrace{ E(0)+\|\omega^-_0\|^2_{L^2(\Omega^-)}}_{M_0}) +C_4\int_{\Gamma_{2,\alpha}} |u^-|^2\ dx_1\,.
\end{equation}
Remembering that from (\ref{3112.9}) we can take for $\alpha$ any value between $\f{H}{4}$ to $\f{H}{2}$, we get by integrating (\ref{3112.11}) for $x_2$ between $\f{H}{4}$ and $\f{H}{2}$ (keeping in mind that $$ \forall\alpha\in[\f{H}{4},\f{H}{2}]\,,\ \int_{\{x_2\ge\alpha\}\cap\p\Omega}   |u^-|^2\ dl\ge \int_{\{x_2\ge\f{H}{2}\}\cap\p\Omega}   |u^-|^2\ dl\ ):$$
\begin{align*}
\f{H}{4}\int_{\{x_2\ge\f{H}{2}\}\cap\p\Omega}   |u^-|^2 dl\le & \frac{C_3 H}{4}  M_0 +C_4 \int_{\Omega^-(t)\cap\{\f{H}{4}\le x_2\le \f{H}{2}\}} |u^-|^2 dx\\
\le & \frac{C_3 H}{4}  M_0 +C_4\|u^-\|_{L^2(\oftp)}^2
\le  C_5 M_0\,.
\end{align*}
Therefore,
\begin{equation}
\label{e40}
\int_{\{x_2\ge\f{H}{2}\}\cap\p\Omega}   |u^-|^2\ dl\le
  C_6 M_0\,.
\end{equation}
We now define 
\begin{equation}
\label{e41}
\Omega_{1,\alpha}=\{x\in\Omega;\ x_1\ge \alpha\}\,,
\end{equation}
and 
\begin{equation}
\label{e42}
\Gamma_{1,\alpha}=\{x\in\Omega;\ x_1= \alpha\}\,.
\end{equation}

From our relation (\ref{3112.9}), we know that for $\alpha> \f{L}{4}$, $\Omega_{1,\alpha}$ does not intersect $\Omega^+(t)$ for all time of existence, and we will work with such values of $\alpha$ in what follows.  

By proceeding as for $\Omega_{2,\alpha}$ we obtain in a verbatim way that
\begin{equation}
\label{e43}
\int_{\{x_1\ge\f{L}{2}\}\cap\p\Omega}   |u^-|^2\ dl\le
   C_7 M_0\,.
\end{equation}
Due to our symmetry in $x_1$, the same estimate holds for 
$
\int_{\{x_1\le-\f{L}{2}\}\cap\p\Omega}   |u^-|^2\ dl$.

Using now our assumption (\ref{3112.12}) on $\Omega$,  we infer from (\ref{e43}) and (\ref{e40}) that
\begin{equation*}
\int_{\Gamma_1^c}   |u^-|^2\ dl\le
  C_{9} M_0\,.
\end{equation*}
Taking $h$, and the $L^2$ norm of velocities as well as the $L^2$ norm of the initial vorticity small enough so that
\begin{equation}
\label{3112.1}
 C_{9} (E(0)+ \|\omega^-_0\|^2_{L^2(\Omega^-)}) \le \f{1}{4}\f{\rho^+-\rho^-}{\rho^- (1+\alpha_\Omega)} g |\Omega^+|\,,
\end{equation}
then provides the desired estimate (\ref{e17}), which concludes our proof of finite in time singularity formation. 

Therefore for $T_{max}$ estimated by (\ref{e26}), we have established that so long as a smooth non self-intersecting and non contacting with $\p\Omega$ solution  exists, blow-up of $f$ will occur at $T_{max}$, namely we proved Theorem 1.

\section{Blow-up of a lower norm if finite-time self-contact or contact with $\p\Omega$}
\label{section9}

We now establish that the finite-time contact cases 2) of Theorem 1 lead to blow-up of a lower norm.

\begin{proof}
We just provide the proof of the more difficult case of self-contact of $\p\ostp$ with itself at $T_{max}$, the other case having a similar proof. Assume that $\p\ostp$ self-intersects at $T_{max}$ and that there exists $C_0>0$ finite so that
\begin{equation}
\label{200417.1}
\forall t\in [0,T_{\max})\,,\ \|\nabla_\tau\tau\|_{L^\infty(\Gamma(t))}+\int_0^t\|\nabla u^-\|_{L^\infty(\Omega^-(t))}\le C_0\,.
\end{equation}

From the fact that the length of the interface and the $L^2$ norm of the velocities $u^\pm$ are bounded, it is not difficult to infer from (\ref{200417.1})  that there exists $C_1>0$ such that
\begin{equation}
\label{200417.2}
\forall t\in [0,T_{\max})\,,\ \int_0^t\| u^-\|_{L^\infty(\Omega^-(t))}\le C_1\,.
\end{equation}

From (\ref{200417.2}) we can define by continuity in time as $t\rightarrow T_{max}$ $$\eta^-(x,T_{\max})=x+\int_0^{T_{max}} v^-(x,t)\ dt\,,$$ since $v^-$ has the same $L^\infty$ norm as $u^-$.
 The self intersection assumption 2) simply means that there exists $x_0\ne x_1$ points of $\p\Omega^+$ such that
  \begin{equation}
  \label{200417.3}
  \eta^-(x_0,T_{max})=\eta^-(x_1,T_{max})\,.
  \end{equation}

In \cite{CoSh2016} we proved there can only be a finite number of additional points $x_i\in\p\Omega^+$ such that $\eta^-(x_i,T_{max})=\eta^-(x_0,T_{max})$. Note that although the assumptions about regularity in \cite{CoSh2016} are stronger than the ones involved here, in order to prove this statement (and the other statements we will make after), it is only the fact that the length of $\Gamma(t)$ stays bounded, as well as the uniform bounds (\ref{200417.1}) and (\ref{200417.2}) which are needed. 

\begin{remark}
The present work does not exclude the possibility that $u^+$ would remain smooth all the way until contact. This exclusion was done in \cite{CoSh2016}. In order to exclude this situation, which corresponds to the case of a splash singularity (in order to have an analogous of the one-phase problem, all relevant norms in one of the phases are assumed bounded), the extra regularity in the framework of \cite{CoSh2016} are needed.
\end{remark}

From (\ref{200417.1}) we see that the tangent vector at $\eta(x_0,t)$ is a continuous function of space (due to the control of $\nabla_\tau \tau$) and time (due to the control of $\nabla u^-$). Given the fact the curve first self-intersect at time $T_{max}$ we have that the tangent vector on $\p\Omega^+(t)$ at each $\eta(x_i,T_{max})$ is the same, and we call it $e_1$ (it is not necessarily horizontal). 

By proceding in a way similar as in Section 6 of \cite{CoSh2016}, we have by the fundamental theorem of calculus applied in a vertical path and a path alongside the interface that:
\begin{align}
\left|\f{d}{dt}(\eta^-(x_0,t)-\eta^-(x_1,t))\right|&=|u^-(\eta^-(x_0,t),t))-u^-(\eta^-(x_1,t),t)|\n\\
&\le C_2 \|\nabla u^-\|_{L^\infty(\Omega^-(t))} |\eta^-(x_0,t)-\eta^-(x_1,t)|\,.
\label{210417.1}
\end{align}
Therefore, 
\begin{equation*}
\f{d}{dt}|\eta^-(x_0,t)-\eta^-(x_1,t)|^2\ge -2  C_2 \|\nabla u^-\|_{L^\infty(\Omega^-(t))} |\eta^-(x_0,t)-\eta^-(x_1,t)|^2\,,
\end{equation*}
which provides by integration:
\begin{equation*}
0=\f{|\eta^-(x_0,T_{max})-\eta^-(x_1,T_{max})|^2}{|\eta^-(x_0,0)-\eta^-(x_1,0)|^2}\ge e^{-2 C_2\int_0^{T_{max}}   \|\nabla u^-\|_{L^\infty(\Omega^-(t))}  \ dt}\,,
\end{equation*}
and thus
\begin{equation}
\label{210417.2}
\int_0^{T_{max}}   \|\nabla u^-\|_{L^\infty(\Omega^-(t))}dt=\infty\,,
\end{equation}
which is in contradiction with (\ref{200417.1}), and establishes  Theorem 2.
\end{proof}


\section{Equations of the rigid body moving inside an inviscid fluid}
\label{section10}

We now consider a {rigid body moving in the inviscid fluid}. The rigid body dynamics is described by the following unknowns:\hfill\break
$\bullet$ The position of the center of the rigid body at time $t$: {$x^s(t).$}\hfill\break
$\bullet$ The angular velocity of rigid body at time $t$: {$r(t)$}.\hfill\break
$\bullet$ The velocity field in the rigid body {$\Omega^s(t)=\Omega^s(0)+(x^s(t)-x^s(0))$} at time $t$:
{$$u^s(x,t)=\frac{d x^s}{dt} (t)+ r(t) (x-x^s(t))^\perp=v^s (t)+ r(t) (x-x^s(t))^\perp\,,$$}
with $(x_1,x_2)^\perp=(-x_2,x_1)\,.$\hfill\break
 $\bullet$ The fluid phase is described by the {incompressible Euler equations in $\Omega^f(t)=\Omega\cap\overline{\Omega^s(t)}^c$}.
The unknowns in the fluid phase are the velocity field {$u^f(x,t)$}, and pressure field {$p(x,t)$.}\hfill\break
This classical {interacting fluid-rigid solid system} is written as:
\begin{subequations}
\begin{alignat}{2}
{\rho_f(u^f_t + u^f\cdot \nabla u^f) + \nabla p} &= -{\rho_f\ g\ e_2} &
\text{in \ 
$\Omega^f(t)$}\,,\label{ch-1}\\
 \operatorname{div} u&={ 0} & \text{in \ 
$\Omega^f(t)$}\,, \\
{u^f\cdot n} &={ u^s\cdot n} & \text{on \  $\partial\Omega^s(t)$}\,,\\
{u^f\cdot n} &={0} & \text{on \  $\p\Omega$}\,, \\
{m_s \frac{d v^s}{dt}}&={\int_{\partial\Omega^s(t)} p\ n dl-m_s g\ e_2}\,,\label{ch-2}\\
{I_s \frac{d r}{dt}}&={\int_{\partial\Omega^s(t)} p\ (x-x^s(t))^\perp\cdot n dl}\,,\label{ch-3} \\
{u^f(0)} & ={u_0} &
\text{in \ 
$\Omega^f$}\,,\\
x^s(0)&=x^s_0\,, \ \ v^s(0)=v^s_{0},\ \ r(0)=r_0\,.
\end{alignat}
\end{subequations}
where $n$ is the exterior unit normal to $\Omega^f(t)$, pointing inside $\Omega^s(t)$, and $e_2$ is the unit vertical vector pointing upwards. Also, $m_s=\rho^s|\Omega^s$ is the mass of the rigid body, and $I_s$ the inertial moment. In this paper we assume that:
\begin{equation}
\label{density}
\rho_s>\rho_f\,.
\end{equation}

Existence and uniqueness to this system (if the initial data satisfies $u^f_0\cdot n=(v^s_0+r_0(x-x^s_0)^\perp)\cdot n$ on $\p\Omega^s$ and $u^f_0\cdot n=0$ on $\p\Omega$) was established by Glass and Sueur in \cite{GlSu2015}, which shows existence and uniqueness of a solution to this problem so long as $\p\ost$ does not touch $\p\Omega$.

\section{Choice of initial data}
\label{section11}
\subsection{Symmetric data}
We denote by $\Omega$ a domain of class $C^1$ locally on one side of its connected boundary, which satisfies the same type of symmetry assumptions as in Section \ref{section6}.

We then choose $\Omega^s$ such that $\overline{\Omega^s}\subset\Omega$ to be an equally symmetric domain with respect to the vertical axis $x_1=0$, which is of the same regularity class as $\Omega$ and is locally on one side of its boundary. 

We then define the initial fluid domain $\Omega^f=\Omega\cap(\overline{\Omega^s})^c$.

We choose 
\begin{subequations}
\begin{align}
v^s(0)=&(0,v^s_2(0))\,,\ \text{with}\ v^s_2(0)< 0\,,\label{3112.4}\\
u^f(0)\ &\text{divergence free with}\ u^f_1(0)\ \text{ odd, } u^f_2(0)\ \text{ even, and}\,,\\
v^s_2(0) n_2=& u^f(0)\cdot n\ \text{on}\ \p\Omega^s\,,\\
r(0)=&0\,,\\
x^s(0)=&(0,h)\,,
\end{align}
\end{subequations}
with $h>0$ such that $\overline{\Omega^s}\subset\Omega$.


Given the symmetry of $\Omega$ and $\Omega^s$ with respect to the $x_1=0$ axis, as well as the symmetry of the initial data with respect to this axis, we then have that for all time of existence $u^f$ and $u^s$ are symmetric with respect to the vertical axis $x_1=0$: $u_1^f(-x_1,x_2)=-u_1^f(x_1,x_2)$ and $u_2^f(-x_1,x_2)=u_2^f(x_1,x_2)$ and 
$
v_1^s(t)=0$, $
r(t)=0$ 
 for all time of existence. Therefore, the rigid solid falls in a vertical translation (at a speed dependent of time) and there is no rotation. The argument is simply to use the construction of solutions of \cite{GlSu2015} pages 937-942, set up with $r=0$ in the functional framework.
 
 
 As we will see later on, the assumption (\ref{3112.4}) together with a small square integrable vorticity 
 ensures that if the rigid body falls from its initial position, then the rigid body keeps falling for all time of existence. It furthermore stays away from $\p\Omega\cap\Gamma_1^c$ by a strictly positive distance 

 We will establish in Section \ref{section8} that the rigid body indeed falls from its initial position later on, which will ensure it stays away from the complementary of the graph bottom part $\Gamma_1$ of $\p\Omega$ by at least $D$. Here, since we know in advance where contact would occur (vertical fall of a body keeping its shape), we just need $\Gamma_1$ to be under the form of a graph where contact would occur.

 

\section{Stream function and conservation of energy}
\label{section12}

\subsection{Stream function}
Since $u^f$ is divergence free we have $u^f=\nabla^\perp\phi=(-\f{\p\phi}{\p x_2},\f{\p\phi}{\p x_1})$, with $\phi$ solution of the elliptic system:
\begin{subequations}
\label{1512.1}
\begin{align}
\Delta\phi(\cdot,t)=&\omega(\cdot,t)=\text{curl}u^f(\cdot,t)\,,\ \text{in}\ \Omega^f(t)\,,\label{1512.1a}\\
\phi(\cdot,t)=& 0\,,\ \text{on}\ \p\Omega\,,\label{1512.1b}\\
\phi(x,t)=& v^s_2(t)x_1\,,\ \text{on}\ \p\Omega^s(t)\,.\label{1512.1c}
\end{align}
\end{subequations} 

Since $\nabla_\tau\phi=u\cdot n$ we then have $\nabla_\tau\phi=0$ on $\p\Omega$ which ensures we can choose $\phi=0$ on the connected $\p\Omega$. On the other hand, we have on $\p\Omega_s(t)$ $$\nabla_\tau\phi=v^s_2(t) n_2=v^s_2(t)\tau_1=v^s_2(t) \nabla_\tau x_1\,,$$ which provides $\phi(x,t)=v^s_2(t) x_1+c(t)$. Next, by the fundamental theorem of calculus, if we denote by $r_2$ the distance from the centre of gravity to the lowest point on $x_1=0$ (which is not necessarily the lowest point of the rigid body, just the lowest on the vertical axis of symmetry):
\begin{align*}
\phi(0,x_2(t)-r_2,\cdot)=&\phi(0,0,\cdot)+\int_0^{x_2(t)-r_2} \f{\p\phi}{\p x_2}(0,x_2,\cdot)\ dx_2\n\\
=& -\int_0^{x_2(t)-r_2} u^f_1(0,x_2,\cdot)\ dx_2=0\,,
\end{align*} due to the fact $u^f_1$ is odd. This in turn provides us with $c(t)=0$ and (\ref{1512.1c}).

\subsection{Energy conservation}

For all time of existence it is classical that the quantity
\begin{equation}
\label{c1}
\f{1}{2}m_s |v^s|^2(t)+\f{1}{2}\rho^f\int_{\oft} |u^f(x,t)|^2\ dx+ m_s g x^s_2(t)+\rho^f g\ \int_{\Omega^f}\eta^f_2\ dx\,,
\end{equation}
is independent of time. Similarly as when establishing (\ref{ff5}), this shows the total energy
\begin{align}
E(t)=&\f{1}{2}m_s |v^s|^2(t)+\f{1}{2}\rho^f\int_{\oft} |u^f(x,t)|^2\ dx+ m_s g x^s_2(t)-\rho^f g x^s_2(t) |\Omega^s|\n\\
=&\f{1}{2}m_s |v^s|^2(t)+\f{1}{2}\rho^f\int_{\oft} |u^f(x,t)|^2\ dx+ \underbrace{(\rho^s-\rho^f) }_{> 0} g x^s_2(t) |\Omega^s|\,,
\label{c2}
\end{align}
is constant in time for all time of existence of a smooth solution (namely from \cite{GlSu2015} so long as no eventual collision with the boundary occurs).

Thus,
\begin{equation}
\f{1}{2}m_s |v^s|^2(t)+\f{1}{2}\rho^f\int_{\oft} |u^f(x,t)|^2\ dx+ (\rho^s-\rho^f)  g x^s_2(t) |\Omega^s|= E(0)\,.
\label{ch1}
\end{equation}
 Moreover, since $x^s_2> 0$, we also have from (\ref{ch1}) the control of the  kinetic energy:
\begin{equation}
\f{1}{2}m_s |v^s|^2(t)+\f{1}{2}\rho^f\int_{\oft} |u^f(x,t)|^2\ dx\le  E(0)\,.
\label{ch3}
\end{equation}

\section{Elliptic estimate away from the contact zone and non zero velocity}
\label{section13}

Our starting point is (\ref{28.10}) which is valid for this problem as well, since it was established from (\ref{ef2}) which is satisfied for this problem as well. 

We have seen earlier on that any point of $\p\Omega\cap \Gamma_1^c$ will stay away from $\ost$ by a positive distance $D>0$ for all time in $[0,T_{max})$, $T_{max}$ being the maximal time of existence of a smooth solution (that we do not assume finite or not here). In a manner similar as we proved the boundary estimate (\ref{3112.11}) for the vortex sheet problem, we have by using $\xi^2(x)  \nabla_{n(x_0)}\phi(x,t)$ (where $\xi$ is a cut-off function in a neighborhood of $x_0\in\Gamma_1^c$) as a test function in the same elliptic system (\ref{1512.1}) that
:
\begin{lemma}
For all time of existence of a smooth solution,
\begin{equation}
\label{reg1}
\|u^f\|^2_{L^2(\p\Omega\cap \Gamma_1^c)}\le C\ (\|u^f\|^2_{L^2(\oft)}+\|\omega_0\|^2_{L^2(\Omega^f)})\,.
\end{equation}
where $C>0$ is independent of time.
\end{lemma}
\begin{remark}
Of course the energy estimate implies that $\|u^f\|_{L^2(\oft}$ is bounded uniformly in time, implying that the right-hand side of (\ref{reg1}) can be replaced by just a constant $C$ independent of time. Although having just $C$ is enough for most of our purposes, it turns out that the more precise form (\ref{reg1}) is used in Section \ref{section14} in a crucial way. 
\end{remark}

In a similar way we also have for $\Gamma_1^s(t)$ being the vertical projection of $\Gamma_1$ on $\p\ost$:

\begin{lemma}
For all time of existence of a smooth solution,
\begin{equation}
\label{reg2}
\|u^f\|_{L^2((\Gamma_1^s(t))^c\cap\p\Omega^s(t))}\le C\,,
\end{equation}
where $C>0$ is independent of $t\in [0,T_{max})$.
\end{lemma}

We next establish that the rigid body keeps falling for all time of existence of a smooth solution with our small square integrable curl assumption.

\begin{lemma}
\label{fall}
With our choice of initial data in Section \ref{section11}, for all time $t>0$ of existence of a smooth solution, we have $v^s_2(t)< 0$.
\end{lemma}
\begin{proof}

Since $v^s_2(0)<0$, we know that for some time $T>0$ we will have
$v^s_2(t)<0$ for all $t<T$. Now let us assume that there exists a first value of  $t_0>0$ such that 
\begin{equation}
\label{1512.2}
v^s_2(t_0)=0\,,
\end{equation}
while there is no contact with $\p\Omega$ at $t_0$, 
with
\begin{equation}
\label{1512.4}
\forall t\in [0,t_0)\,,\ v^s_2(t)<0\,.
\end{equation}
(namely the rigid body is with zero speed at $t_0$, and does not touch $\p\Omega$, and was before that time falling at a negative vertical speed). From the start of this Section, we have $u^f=\nabla^\perp\phi$ satisfying (\ref{1512.1}).

From the elliptic system (\ref{1512.1}) we immediately have by Green's theorem:
\begin{align}
\int_{\oft}|u^f|^2\ dx&=-\int_{\oft}\omega\phi\ dx+v^s_2(t)\int_{\p\ost}\nabla_n\phi\ x_1\ dl\n\\
&=-\int_{\oft}\omega\phi\ dx-v^s_2(t)\int_{\p\ost}u^f\cdot\tau\ x_1\ dl\,.\label{29.01}
\end{align}
Therefore,
\begin{equation}
\label{29.02}
\|u^f\|^2_{L^2(\oft)}\le \|\phi\|_{L^2(\oft)}\|\omega_0\|_{L^2(\Omega^f)}+|v^s_2(t)| \left|\int_{\p\ost}u^f\cdot\tau\ x_1\ dl\right|\,.
\end{equation}
We now need to establish a Poincar\'e inequality for $\phi$ (independent of how close to contact we are), in order to control $\|\phi\|_{L^2(\oft)}$. To do so we simply notice that if we define $\bar{\phi}$ as 
\begin{equation}
\label{1403.2}
\bar{\phi}(x,t)=1_{\Omega^f(t)}(x)\phi(x)+1_{\overline{\ost}}(x)v^s_2(t) x_1\,,
\end{equation}
we have due to the continuity (\ref{1512.1c}) that $\bar{\phi}\in H^1(\Omega)$, and due to (\ref{1512.1b}) that $\bar{\phi}\in H^1_0(\Omega)$. Note here that this is done for any $t$ such that $\p\ost$ and $\p\Omega$ do not intersect.

By the standard Poincar\'e inequality for $\bar{\phi}$ in $\Omega$, we then have (independently of any $t$ such that $\p\ost$ and $\p\Omega$ do not intersect):
\begin{equation}
\label{1403.3}
\int_{\Omega}\bar{\phi}^2\ dx\le C_{\Omega} \int_{\Omega}|\nabla\bar{\phi}|^2\ dx=C_{\Omega} \left(\int_{\oft}|\nabla{\phi}|^2\ dx+v^s_2(t)^2 |\ost|\right)\,.
\end{equation}
From (\ref{29.02}) and (\ref{1403.3}) we infer successively:
\begin{align*}
\|u^f\|^2_{L^2(\oft)}&\le \f{\|\phi\|^2_{L^2(\oft)}}{2C_\Omega}+\f{C_\Omega}{2}\|\omega_0\|^2_{L^2(\Omega^f)}+|v^s_2(t)| \left|\int_{\p\ost}u^f\cdot\tau x_1 dl\right|\\
&\le \f{1}{2}\|\nabla\phi\|^2_{L^2(\oft)}+\f{1}{2} v^s_2(t)^2 |\Omega^s|+\f{C_\Omega}{2}\|\omega_0\|^2_{L^2(\Omega^f)}\\
& \ \ \ +|v^s_2(t)| \left|\int_{\p\ost}u^f\cdot\tau\ x_1\ dl\right|\,.
\end{align*}
Therefore,
\begin{equation}
\label{1403.4bis}
\|u^f\|^2_{L^2(\oft)}\le  v^s_2(t)^2 |\Omega^s|+{C_\Omega}\|\omega_0\|^2_{L^2(\Omega^f)} +2|v^s_2(t)| \left|\int_{\p\ost}u^f\cdot\tau\ x_1\ dl\right|\,.
\end{equation}
We will need later on to replace the integral set on $\p\ost$ by an integral set on $\p\Omega$. This is done in the following way:
\begin{equation}
\label{28.03bis}
\int_{\p\ost} x_1\ u^f\cdot\tau\ dl
=\int_{\p\ost} x_1\ (u^f_1 n_2-u^f_2 n_1)\ dl
\,.
\end{equation}
Now, by integration by parts in $\oft$ for the right-hand side of (\ref{28.03bis}):
\begin{align}
\int_{\p\ost} x_1 u^f\cdot\tau dl=&-\int_{\p\Omega} x_1 (u^f_1 n_2-u^f_2 n_1) dl+\int_{\oft} \f{\p(x_1 u^f_1)}{\p x_2}-\f{\p(x_1 u^f_2)}{\p x_1} dx\n\\
=&-\int_{\p\Omega} x_1 u^f\cdot\tau\ dl- \int_{\oft}x_1\omega+ u^f_2\ dx\n\\
=&-\int_{\p\Omega} x_1 u^f\cdot\tau\ dl- \int_{\oft}x_1\omega+ \f{\p\phi}{\p x_1}\ dx\n\\
=&-\int_{\p\Omega} x_1 u^f\cdot\tau dl- \int_{\oft}x_1\omega dx -v^s_2(t)\int_{\p\ost} x_1 n_1  dl\,,
\label{300617.1}
\end{align}
where we used (\ref{1512.1b}), (\ref{1512.1c}) and integration by parts to obtain the last term above.
Therefore using (\ref{300617.1}) in (\ref{1403.4bis}) we obtain:
\begin{align*}
\|u^f\|^2_{L^2(\oft)}\le & v^s_2(t)^2 (|\Omega^s|+2\text{diam}({\Omega^s})|\p\Omega^s|)+{C_\Omega}\|\omega_0\|^2_{L^2(\Omega^f)}\\
& +2|v^s_2(t)| \left|\int_{\p\Omega}u^f\cdot\tau\ x_1\ dl\right|+ 2 |v^s_2(t)|\left|\int_{\oft} x_1\omega\ dx\right|\,.
\end{align*}
Using Young's inequality for the last term of the right-hand side, we obtain:
\begin{align}
\|u^f\|^2_{L^2(\oft)} 
\le & v^s_2(t)^2 (|\Omega^s|+2\text{diam}({\Omega^s})|\p\Omega^s|+\text{diam}(\Omega)^2)\n\\
&+({C_\Omega}+|\Omega^f|)\|\omega_0\|^2_{L^2(\Omega^f)} +2|v^s_2(t)| \left|\int_{\p\Omega}u^f\cdot\tau\ x_1\ dl\right|\n\\
\le & D_\Omega\ (v^s_2(t)^2  +|v^s_2(t)| \left|\int_{\p\Omega}u^f\cdot\tau\ x_1\ dl\right|+\|\omega_0\|^2_{L^2(\Omega^f)})
\,,
\label{1403.4}
\end{align}
with
\begin{equation}
\label{040717.1}
D_\Omega=\max({|\Omega^s|+2\text{diam}({\Omega^s})|\p\Omega^s|+\text{diam}(\Omega)^2},2, C_\Omega+|\Omega^f|)\,.
\end{equation}

Thus, if $v^s_2(t_0)=0$, we infer from (\ref{1403.4}) that
\begin{equation}
\label{29.03}
\|u^f\|^2_{L^2(\Omega^f(t_0))}\le D_\Omega\|\omega_0\|^2_{L^2(\Omega^f)}\,.
\end{equation}
Therefore, the total energy satisfies
\begin{equation}
\label{1512.5}
E(t_0)\le \f{D_\Omega}{2}\rho^f \|\omega_0\|^2_{L^2(\Omega^f)}+(\rho^s-\rho^f)g x^s_2(t_0)|\Omega^s|\,.
\end{equation}
Now, from (\ref{1512.4}) we infer that
\begin{equation}
\label{1512.6}
\forall t\in [0,t_0)\,,\  x^s_2(t)>x^s_2(t_0)\,.
\end{equation}
With our assumption (\ref{smallvortex}) of smallness of $\omega_0$ relative to $v^s_2(0)$, (\ref{1512.5}) and (\ref{1512.6}) lead to
$$E(t_0)< \f{m_s}{2}|v^s_2(0)|^2+(\rho^s-\rho^f) g x^s_2(0)=E(0)\,,$$ which is in contradiction with the conservation of energy. Therefore, for all time $t$ such that $\ost$ does not intersect $\p\Omega$, we have $v^s_2(t)<0$, which proves the lemma.

\end{proof}

We can now prove our general finite-time contact Theorem \ref{theorem3}:

\section{Finite-time contact for the rigid body falling over a bottom $\p\Omega$ under the form of a graph}
\label{section14}
\begin{proof}
We now assume that the rigid body does not touch $\p\Omega$ at any finite $t>0$. From \cite{GlSu2015}, we then know the maximal time of existence of a smooth solution  satisfies 
\begin{equation}
\label{1512.9}
T_{max}=\infty\,.
\end{equation}
 
From Lemma \ref{fall}, we infer
\begin{equation*}
\forall t\ge 0\,,\  x^s_2(t)=h-\int_0^t |v^s_2(s)|\ ds\,,
\end{equation*}
which shows that
\begin{equation}
\label{1512.10}
\int_0^\infty |v^s_2(s)|\ ds\le h\,.
\end{equation}
We now integrate (\ref{28.10}) from $0$ to $t$:
\begin{align}
m_s\ (v_2^s(t)-v_2^s(0))
=&\rho^f\int_0^t \int_{\p\Omega}\f{|u^f|^2}{2}\ n_2\ dl\ dt -(m_s-\rho^f|\Omega^s|) gt\n\\
&- \rho^f\int_{\p\Omega} x_1 u^f(\cdot,t)\cdot\tau\ dl+\rho^f\int_{\p\Omega} x_1 u^f(\cdot,0)\cdot\tau\ dl\n\\
&- \rho^f\int_{\oft}u_2^f(\cdot,t)\ dx + \rho^f\int_{\Omega^f}u_2^f(\cdot,0)\ dx  
\,.\label{0604.1}
\end{align}
We now write 
\begin{equation}
\label{0604.2}
\int_{\p\Omega}\f{|u^f|^2}{2}\ n_2\ dl=\int_{\Gamma_1}\f{|u^f|^2}{2}\ n_2\ dl +
\int_{\Gamma_1^c\cap\p\Omega}\f{|u^f|^2}{2} (-\alpha_\Omega+ (n_2+\alpha_\Omega))\ dl\,,
\end{equation}
which  thanks to $\Gamma_1\subset\{n_2\le-\alpha_\Omega <0\}$, $0<\alpha_\Omega\le 1$, provides us with
\begin{equation}
\label{0604.3}
\int_{\p\Omega}\f{|u^f|^2}{2}\ n_2\ dl\le -\alpha_\Omega \int_{\p\Omega}\f{|u^f|^2}{2}\ dl +
\int_{\Gamma_1^c\cap\p\Omega}{|u^f|^2}\ dl\,.
\end{equation}
Using our elliptic estimate (\ref{reg1}) away from the contact zone, together with the other elliptic estimate (\ref{1403.4}) we infer that
\begin{align}
\int_{\p\Omega}\f{|u^f|^2}{2}\ n_2 dl\le & -\alpha_\Omega\int_{\p\Omega}\f{|u^f|^2}{2}\ dl+ C{D_\Omega} (\|\omega_0\|^2_{L^2(\Omega^f)}+
v_2^s(t)^2 \n\\
&+|v^s_2(t)| \left|\int_{\p\Omega}u^f\cdot\tau\ x_1\ dl\right|)+ C \|\omega_0\|^2_{L^2(\Omega^f)}\n\\
 \le & -\alpha_\Omega\int_{\p\Omega}\f{|u^f|^2}{2}\ dl +(C{D_\Omega}+C)|\|\omega_0\|^2_{L^2(\Omega^f)}\n\\
& +C D_\Omega \left(
v_2^s(t)^2 +  |v^s_2(t)| \text{diam}(\Omega)\sqrt{|\p\Omega|}\sqrt{\int_{\p\Omega}|u^f\cdot\tau|^2 dl}\right)\,.
\label{0604.4}
\end{align}
Noticing that $$\int_0^t (v^s_2)^2\ dt\le \|v_2^s\|_{L^\infty(0,t)}
\int_0^t |v_2^s(t)|\ dt\le h \|v_2^s\|_{L^\infty(0,t)}$$ due to (\ref{1512.10}), we infer from (\ref{0604.4}) that for $\epsilon>0$ :
\begin{align*}
\int_0^t\int_{\p\Omega}\f{|u^f|^2}{2}\ n_2\ dl\ dt \le & -\alpha_\Omega\int_0^t \int_{\p\Omega}\f{|u^f|^2}{2}\ dl\ dt + (C{D_\Omega}+C) \|\omega_0\|^2_{L^2(\Omega^f)}\ t\n\\
& +C D_\Omega  \|v_2^s\|_{L^\infty(0,t)}
\int_0^t |v_2^s(t)|\ dt \n\\
&\  +\f{C D_\Omega \text{diam}(\Omega)^2}{4 \epsilon} \int_0^t v^s_2(t)^2\ dt\\
& +C D_\Omega\epsilon|\p\Omega|\int_0^t\int_{\p\Omega}|u^f|^2\ dl\ dt\,.
\end{align*}
Thus, for $\displaystyle\epsilon=\f{\alpha_\Omega}{4C D_\Omega|\p\Omega|}$  this provides us with
\begin{align}
\int_0^t\int_{\p\Omega}\f{|u^f|^2}{2}\ n_2\ dl\ dt & \le -\alpha_\Omega\int_0^t \int_{\p\Omega}\f{|u^f|^2}{4}\ dl\ dt + (C{D_\Omega}+C) \|\omega_0\|^2_{L^2(\Omega^f)}\ t\n\\
& +C D_\Omega (1+\f{CD_\Omega|\p\Omega|\text{diam}(\Omega)^2}{\alpha_\Omega}) \sqrt{\frac{2E(0)}{m_s}}h \,,
\label{0604.5}
\end{align}
where we used the conservation of total energy and (\ref{1512.10}) to obtain (\ref{0604.5}) from the previous inequality.
Reporting (\ref{0604.5}) in (\ref{0604.1}) we obtain:
\begin{align}
\rho^f\int_{\p\Omega}x_1 u^f(\cdot,t)\cdot\tau\ dl\le & -\rho^f\alpha_\Omega\int_0^t \int_{\p\Omega}\f{|u^f|^2}{4}\ dl\ dt\n\\
& + \rho^f C( D_\Omega+1) t \|\omega_0\|^2_{L^2(\Omega^f)}-(\rho^s-\rho^f)|\Omega^s|gt+C_1\,,
\label{0604.6}
\end{align}
where we also used the conservation of energy to have an estimate uniform in time for the terms not explicitly reported in (\ref{0604.6}) and controlled by $C_1>0$ independent of time.
Therefore, remembering our small curl assumption (\ref{smallvortex}),
\begin{equation*}
\rho^f \|x_1\|_{L^\infty(\p\Omega)}\int_{\p\Omega}|u^f|(\cdot,t) dl\ge \frac{\alpha_\Omega \rho^f}{4}\int_0^t \int_{\p\Omega}|u^f|^2 dl dt +\f{(\rho^s-\rho^f)|\Omega^s|g}{2} t-{C_1}\,.
\end{equation*}
Proceeding in an identical manner as we obtain (\ref{e26}) from (\ref{e23bis}) we infer that the maximal time of existence $T_{max}$ of a smooth solution is finite, which is in contradiction with our assumption that it was infinite. Therefore,  
\begin{equation}
\label{1512.23}
T_{max}<\infty\,.
\end{equation}
From \cite{GlSu2015}, the rigid body will then touch $\p\Omega$ at $T_{max}$. \qed
\end{proof}

From now on we assume that $\omega_0=0$.

\section{Equivalence of norms for the velocity field when $\omega=0$}
\label{section15}
We have by integration by parts:
\begin{align*}
0=&\int_{\Omega^f(t)}\triangle u^f\cdot u^f\ dx\\
=&-\int_{\Omega^f(t)}|\nabla u^f|^2\ dx+\int_{\partial\Omega^f(t)} \nabla_n u^f\cdot u^f\ dl\,.
\end{align*}
Therefore, expanding in the $(\tau,n)$ basis,
\begin{equation*}
0=-\int_{\Omega^f(t)}|\nabla u^f|^2\ dx
+\int_{\partial\Omega^f(t)} \nabla_n u^f\cdot n\ u^f\cdot n+\nabla_n u^f\cdot\tau\ u^f\cdot\tau\ dl\,,
\end{equation*}
and using the divergence and curl relations on the boundary integral,
\begin{equation*}
0=-\int_{\Omega^f(t)}|\nabla u^f|^2\ dx
+\int_{\partial\Omega^f(t)} -\nabla_\tau u^f\cdot \tau\ u^f\cdot n+(\nabla_\tau u^f\cdot n)\ u^f\cdot\tau dl\,.
\end{equation*}
Rearranging the last term of the boundary integral, this identity yields:
\begin{align}
0=&-\int_{\Omega^f(t)}|\nabla u^f|^2\ dx-\int_{\partial\Omega^f(t)} \nabla_\tau u^f\cdot \tau\ u^f\cdot n dl\n\\
&+\int_{\partial\Omega^f(t)} (\nabla_\tau (u^f\cdot n) -u^f\cdot\nabla_\tau n)\ u^f\cdot\tau\ dl\label{en1.d}\,.
\end{align}
By integrating by parts the first integral set on $\p\oft$, (\ref{en1.d}) becomes:
\begin{align}
\int_{\Omega^f(t)}|\nabla u^f|^2 dx=
&\int_{\partial\Omega^f(t)} u^f\cdot \nabla_\tau \tau\ u^f\cdot n + u^f\cdot\tau \ \nabla_\tau\ (u^f\cdot n) \ dl\n\\
&+\int_{\p\oft} \nabla_\tau (u^f\cdot n)\ u^f\cdot\tau\ - u^f\cdot \nabla_\tau n\ u^f\cdot\tau\ dl\nonumber\\
=
&\int_{\partial\Omega^f(t)} \kappa u^f\cdot n \ u^f\cdot n + u^f\cdot\tau  \nabla_\tau\ (u^f\cdot n)  dl\n\\
&+\int_{\p\oft} \nabla_\tau (u^f\cdot n) u^f\cdot\tau\ +\kappa u^f\cdot \tau u^f\cdot\tau dl\,,\label{en2.0}
\end{align}
where we used $\nabla_\tau\ \tau=\kappa\ n$ and $\nabla_\tau\ n=-\kappa\tau$ on $\p\oft$ (we remind $n$ points outside $\oft$).
Using the boundary conditions $u^f\cdot n=v_s(t)\cdot n$ on $\p\ost$ and $u^f\cdot n=0$ on $\p\Omega$, we obtain from (\ref{en2.0}):

\begin{align}
\int_{\Omega^f(t)}|\nabla u^f|^2\ dx
=&\int_{\partial\Omega_s(t)} \kappa\ (v_s\cdot n)^2+ 2 u^f\cdot\tau \ v_s\cdot\nabla_\tau n +\kappa\ (u^f\cdot\tau)^2 \ dl\n\\
& + \int_{\partial\Omega}\kappa\ (u^f\cdot\tau)^2 \ dl\label{en2.b}\,.
\end{align}

\section{A formula for acceleration at time of contact for the case without vorticity}
\label{section16}
Using (\ref{29.01}) we obtain:
\begin{equation*}
\int_{\oft}|u^f|^2\ dx=-v^s_2(t)\int_{\p\ost}u^f\cdot\tau\ x_1\ dl\,,
\end{equation*}
Using $u^f=\nabla\phi$ for some potential $\phi$ (since in this Section $u^f$ is curl free), this provides:
\begin{align}
\int_{\oft}|u^f|^2\ dx&=-v^s_2(t)\int_{\p\ost}\nabla_\tau\phi\ x_1\ dl\n\\
&=v^s_2(t)\int_{\p\ost}\phi\ \nabla_\tau x_1\ dl\n\\
&=v^s_2(t)\int_{\p\ost}\phi\ \tau_1\ dl\n\\
&=v^s_2(t)\int_{\p\ost}\phi\ n_2\ dl\,.
\label{1303.1}
\end{align}
We will also need the following simple identity:
\begin{align}
\int_{\oft}u^f_2\ dx&=\int_{\oft} \f{\p\phi}{\p x_2}\ dx\n\\
&=\int_{\p\Omega} \phi\ n_2\ dl+\int_{\p\Omega_s(t)} \phi\ n_2\ dl\,.
\label{1303.2}
\end{align}

From (\ref{28.10}), we obtain:
\begin{align}
m_s\ \frac{dv^s_2}{dt}
=&\rho^f\int_{\p\Omega}\f{|u^f|^2}{2}\ n_2\ dl-\rho^f \frac{d}{dt}\int_{\oft} u^f_2\ dx -(m_s-\rho^f|\Omega_s|) g\n\\
&+ \rho^f\f{d}{dt} \int_{\p\Omega}\phi\ n_2\ dl\,,
\label{1612.2}
\end{align}
where we used $\displaystyle\int_{\p\Omega} u^f\cdot\tau x_1\ dl=-\displaystyle\int_{\p\Omega}\phi n_2\ dl$ which is established similarly as in the proof of (\ref{1303.1}).
Using (\ref{1303.2}) in (\ref{1612.2}) yields:
\begin{equation}
m_s\ \frac{dv^s_2}{dt}
=\rho^f\int_{\p\Omega}\f{|u^f|^2}{2}\ n_2\ dl\  -(m_s-\rho^f|\Omega_s|) g- \rho^f\f{d}{dt} \int_{\p\ost}\phi\ n_2\ dl\,.
\label{1612.3}
\end{equation}
Using (\ref{1303.1}) in (\ref{1612.3}) yields:
\begin{equation}
m_s\ \frac{dv^s_2}{dt}
=\rho^f\int_{\p\Omega}\f{|u^f|^2}{2}\ n_2\ dl\  -(m_s-\rho^f|\Omega_s|) g- \rho^f\f{d}{dt}\left(\f{\|u^f\|^2_{L^2(\oft)}}{v^s_2}\right)\,.
\label{1612.4}
\end{equation}
From our conservation of energy,
 (\ref{1612.4}) becomes:
\begin{align*}
m_s\ \frac{dv^s_2}{dt}
=&\rho^f\int_{\p\Omega}\f{|u^f|^2}{2}\ n_2\ dl\  -(m_s-\rho^f|\Omega_s|) g\\
&- \f{d}{dt}\left(\f{2E(0)-m_s (v^s_2)^2+2(\rho^f-\rho^s)g x^s_2|\Omega^s|}{v^s_2}\right)\,,
\end{align*}
and thus,
\begin{equation}
\label{1612.5bis}
0
=\rho^f\int_{\p\Omega}\f{|u^f|^2}{2}\ n_2\ dl\  -(m_s-\rho^f|\Omega_s|) g- \f{d}{dt}\left(\f{2E(0)+2(\rho^f-\rho^s)g x^s_2|\Omega^s|}{v^s_2}\right)\,.
\end{equation}
Since $n_2<-\alpha_\Omega$ on $\Gamma_1$, we have that $\displaystyle \int_0^{T_{max}}\int_{\Gamma_1} |u^f|^2 n_2 dl$ is well defined in $[-\infty,0]$. Since $u^f$ is bounded away from contact in $L^2(\Gamma_1^c\cap\p\Omega)$ we also have that $\displaystyle \int_0^{T_{max}}\int_{\Gamma_1^c\cap\p\Omega} |u^f|^2 n_2 dl$ is well defined in $\R$. Therefore (\ref{1612.5bis}) shows that
\begin{equation*}
\lim_{t\rightarrow T_{max}^-}\f{2E(0)+2(\rho^f-\rho^s)g x^s_2(t)|\Omega^s|}{v^s_2(t)}\in [-\infty,\infty)\,.
\end{equation*}
Since the coefficient 
\begin{equation}
\label{coef}
2E(0)+2(\rho^f-\rho^s)g x^s_2(t)|\Omega^s\ge 2\underbrace{(\rho^s-\rho^f)}_{>0} g \underbrace{(h-x^s_2(t))}_{>0}|\Omega^s|>0
\end{equation}  is positive this provides the following limit is well-defined:
\begin{equation*}
v^s_2(T_{max})=\lim_{t\rightarrow T_{max}^-} v^s_2(t)\in (-\infty,0]\,,
\end{equation*}
which allows to speak of a velocity at contact.

\section{Blow-up of $L^2(\p\oft)$ norm of the velocity field in the fluid in the case without vorticity}
\label{section17}

\subsection{Blow-up of the $L^2(\p\oft)$ norm of $u^f$ as $t\rightarrow T_{\max}^-$ for the case $v^s_2(T_{max})=0$}
Integrating (\ref{1612.5bis}) from $0$ to $T_{max}$ and using (\ref{coef}) we infer
$$\int_0^{T_{max}} \int_{\p\Omega}|u^f|^2 n_2\ dl=-\infty\,.$$

\subsection{Blow-up of the $L^2(\p\oft)$ norm of $u^f$ as $t\rightarrow T_{\max}^-$ for the case $v^s_2(T_{max})<0$}

\begin{proof}
In this case we have the existence of $\alpha>0$ such that  $$\forall t\in [0,T_{max})\,,\ -\sqrt{\f{2E(0)}{m_s}}\le v^s_2(t)<-\alpha<0\,.$$ 

Let us now assume that we have the existence of $\beta>0$ such that for a sequence of points $t_n$ converging to $T_{max}$ we have
\begin{equation}
\label{100817.1}
\int_{\p\Omega^f(t_n)}|u^f(t_n)|^2\ dl\le \beta\,.
\end{equation}
In the following we work exclusively with this sequence of points, that we denote $t$.

Let us denote by $x_0\in\p\Omega$ a point where intersection occurs at $T_{max}$. By assumption our normal vector satisfies
\begin{equation*}
n_2(x_0)<-\alpha_\Omega<0\,.
\end{equation*}
We also know that at intersection, the direction of the normal vector to $\p\Omega^s(T_{max})$ at $x_0$ will be the same as $n(x_0)$. We now for $\epsilon>0$ small consider the curve $\gamma_\epsilon\subset\p\Omega$ centered at $x_0$, and with length $\epsilon$. For $t$ close to $T_{max}$, we then call $\gamma_\epsilon(t)$ the projection of $\gamma_\epsilon$ on $\p\Omega^s(t)$ parallel to $n(x_0)$. Namely for $t$ close to $T_{max}$ and $\epsilon>0$ small, these two curves are almost like segments of length $\epsilon$ which are orthogonal to $n(x_0)$. 

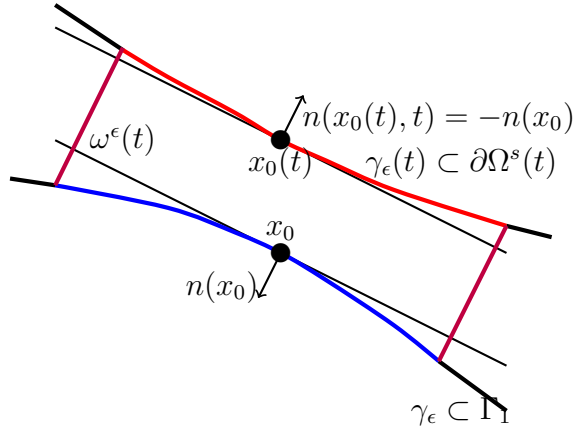
\begin{figure}[h]
\begin{tikzpicture}[scale=0.3]

  \draw[color=black, thick] (-10,5)--(10,-5);
 \draw[color=blue,ultra thick] plot[smooth,tension=.6] coordinates{( -10,3)(-5,2) (-2,0.9) (0,0) (2,-1.2) (5,-3.2) (7,-4.8)};
 \draw[color=black,ultra thick] plot[smooth,tension=.6] coordinates{ (7,-4.8) (10,-7)};
 \draw[color=black,ultra thick] plot[smooth,tension=.6] coordinates{ (-10,3) (-12,3.3)};
   \draw (0,0) node { $\newmoon$};  
\draw[thick,->] (0,0) -- (-1,-2);  
 \draw (-2.6,-1.5) node { {\large $n(x_0)$}};  
  \draw (0,1) node {\large $x_0$};
   \draw (8,-7) node {\large $\gamma_\epsilon\subset\Gamma_1$};
   
     \draw[color=black, thick] (-10,10)--(10,0);
 \draw[color=red,ultra thick] plot[smooth,tension=.6] coordinates{( (-7,9) (-5,7.8) (-2,6.2) (0,5) (2,4.1) (5,2.8)(10,1.2)};
  \draw[color=black,ultra thick] plot[smooth,tension=.6] coordinates{( -10,11)(-7,9)};
   \draw[color=black,ultra thick] plot[smooth,tension=.6] coordinates{( 10,1.2)(12,0.7)};
   \draw (0,5) node { $\newmoon$};  
\draw[thick,->] (0,5) -- (1,7);  
 \draw (7,6) node { {\large $n(x_0(t),t)=-n(x_0)$}};  
  \draw (0,4) node {\large $x_0(t)$};
   \draw (8,4) node {\large $\gamma_\epsilon(t)\subset\p\ost$};
   \draw[color=purple, ultra thick] (7,-4.8)--(10,1.2); 
    \draw[color=purple, ultra thick] (-10,3)--(-7,9.1);
   \draw (-7,5) node {\large $\omega^\epsilon(t)$};
   \end{tikzpicture} 
 \caption{ $\omega^\epsilon(t)$=fluid region between blue, red and purple curves.}\label{fig2}  
\end{figure}

\begin{remark}
$\gamma^\epsilon(t)$ and $\gamma_\epsilon$ do not need to be locally on one side of the tangent at $x_0(t)$ and $x_0$ respectively.
\end{remark}

Since our curves are of class $C^2$ we have the existence of $C>0$ such that the area $A_1$ between $\gamma_\epsilon$ and the tangent line passing through $x_0$ satisfies:
\begin{equation}
\label{100817.3}
|A_1|\le C \epsilon^3\,.
\end{equation}
Next we remember that since the fall of the rigid body is purely vertical, the vertical projection of $x_0$ onto $\p\Omega^s(t)$, that we call $ x_0(t)$ satisfies 
\begin{equation}
\label{100817.10}
n(x_0(t))=-n(x_0)\,,
\end{equation} 
as well as $x_0(t)\in\gamma_\epsilon(t)$ if $t$ is close enough to $T_{max}$. Then similarly, by increasing $C$ if necessary, the area $A_2(t)$ between the tangent line passing though $x_0(t)$ (which is perpendicular to $n(x_0)$) and $\gamma_\epsilon(t)$ satisfies:
\begin{equation}
\label{100817.4}
|A_2(t)|\le C \epsilon^3\,.
\end{equation}
Next the distance between between $x_0$ and $x_0(t)$ satisfies (since $x_0(T_{max})=x_0$):
\begin{equation}
\label{100817.5}
|x_0-x_0(t)|\le \underbrace{\sqrt{\f{2E(0)}{m_s}}}_{C_0}(T_{max}-t)\,.
\end{equation}
If we denote by $\omega_\epsilon(t)$ the region comprised between $\gamma_\epsilon$, $\gamma_\epsilon(t)$, and the two segments parallel to $n(x_0)$ and starting at an extremity point of $\gamma_\epsilon$, we have:
\begin{equation}
\label{100817.6}
|\omega_\epsilon(t)|\le |A_1(t)|+|A_2(t)|+\epsilon |x_0-x_0(t)|\le 2C\epsilon^3+ C_0 (T_{max}-t)\epsilon\,.
\end{equation}
Now by integration by parts,
\begin{equation}
\label{100817.7}
\int_{\omega_\epsilon(t)} \nabla_{n(x_0)} u^f\cdot n(x_0)\ dx=\int_{\gamma_\epsilon\cup\gamma_\epsilon(t)} u^f\cdot n(x_0)\ n(x_0)\cdot n\ dl\,.
\end{equation}
Thus, with (\ref{en2.b}) and our assumption (\ref{100817.1}) we obtain by Cauchy-Schwarz:
\begin{equation}
\label{100817.8}
\left|\int_{\gamma_\epsilon\cup\gamma_\epsilon(t)} u^f\cdot n(x_0)\ n(x_0)\cdot n\ dl\right|\le C\sqrt{|\omega^\epsilon(t)|}\,,
\end{equation}
for some $C>0$ independent of $t$ and $\epsilon$.
We first have on $\gamma_\epsilon$,
\begin{equation}
\label{100817.9}
|n(x_0)-n|\le \max{|\kappa|}|\gamma_\epsilon|\le C \epsilon\,,
\end{equation}
for some $C>0$ independent of $t$ and $\epsilon$.
Taking $\epsilon>0$ and $T_{\max}-t>0$ small enough, we have
\begin{equation}
\label{100817.11bis}
|\gamma_\epsilon(t)|\le 2\epsilon\,.
\end{equation}

Due to $x_0(t)\in\gamma_\epsilon(t)$ for $t$ close enough to $T_{max}$ and (\ref{100817.11bis})
\begin{equation*}
\forall x\in \gamma_\epsilon(t)\,,\ \ |x- x_0(t)|\le 2\epsilon\,.
\end{equation*}
Therefore the distance on $\gamma_\epsilon(t)$ satisfies (for $\epsilon>0$ small enough) that 
\begin{equation*}
\forall x\in \gamma_\epsilon(t)\,,\ \ d_{\gamma_\epsilon(t)}(x, x_0(t))\le 3\epsilon\,.
\end{equation*}
Therefore,
\begin{equation}
\label{100817.13}
\forall x\in\gamma_\epsilon(t)\,,\ | n({x_0}(t))-n(x)|\le  \max{|\kappa|}max_{\gamma_\epsilon(t)} d_{\gamma_\epsilon(t)}(x,\tilde x_0(t))
\le  C \epsilon \,,
\end{equation}
for some $C>0$ independent of $t$ and $\epsilon$.
We now write for $\gamma_\epsilon$:
\begin{equation*}
u^f\cdot n(x_0) n(x_0)\cdot n= u^f\cdot (n(x_0)-n) n(x_0)\cdot n+u^f\cdot n (n(x_0)-n)\cdot n+u^f\cdot n \underbrace{n \cdot n}_{=1}\,,
\end{equation*}
while for $\gamma_\epsilon(t)$,
\begin{equation*}
u^f\cdot n(x_0) n(x_0)\cdot n= u^f\cdot (n(x_0)+n) n(x_0)\cdot n-u^f\cdot n (n(x_0)+n)\cdot n+u^f\cdot n \underbrace{n \cdot n}_{=1}\,.
\end{equation*}
Using these two equations in (\ref{100817.8}), (\ref{100817.9}) and (\ref{100817.13}) we infer that for some $C>0$ independent of $t$ and $\epsilon$
\begin{align}
\left|\int_{\gamma_\epsilon(t)} u^f\cdot n\ dl\right|\le & C\sqrt{|\omega^\epsilon(t)|}+ C\epsilon\int_{\gamma_\epsilon(t)\cup\gamma_\epsilon}|u^f|\ dl\n\\
\le & C\sqrt{|\omega^\epsilon(t)|}+ C\epsilon(\sqrt{|\gamma_\epsilon|}+\sqrt{|\gamma_\epsilon(t)|})\|u^f\|_{L^2(\p\oft)}\,.
\label{100817.14}
\end{align}
 Using the boundary condition $u^f\cdot n=v^s_2 n_2$ on $\p\ost$ and our estimate (\ref{100817.6}) as well as our crucial assumption (\ref{100817.1}) we then obtain from (\ref{100817.14}) that for some $C>0$ independent of $t$ and $\epsilon$:
\begin{equation}
\label{100817.15}
\left|v^s_2(t) \int_{\gamma_\epsilon(t)}  n_2\ dl\right|\le C\sqrt{\epsilon^3+\epsilon(T_{max}-t)}+ C\epsilon\sqrt{\epsilon}\,.
\end{equation}
Therefore, since for $\epsilon>0$ small enough and $t$ close enough to $T_{max}$ $n$ on $\gamma_\epsilon(t)$ is close to $-n(x_0)$ which satisfies $n_2(x_0)<-\alpha_\Omega$, we then infer from (\ref{100817.15}) that
\begin{equation*}
|v^s_2(t)| \alpha_\Omega\epsilon \le 2 C\sqrt{\epsilon^3+\epsilon(T_{max}-t)}+ 2C\epsilon\sqrt{\epsilon}\,.
\end{equation*}
Letting (with $\epsilon>0$ fixed) $t$ converge to $T_{\max}$ we then have
\begin{equation*}
|v^s_2(T_{max})| \alpha_\Omega \le 2 C\sqrt{\epsilon}+ 2C\sqrt{\epsilon}\,.
\end{equation*}
This identity being true for any $\epsilon>0$ small enough we obtain that $v^s_2(T_{max})=0$, in contradiction with our assumption that $v^s_2(T_{max})<0$. Therefore our assumption (\ref{100817.1}) has to be rejected, which means we proved 1) of Theorem \ref{theorem4}:
\begin{equation}
\label{100817.17}
\lim_{t\rightarrow T_{max}^-} \|u^f\|_{L^2(\p\Omega^f(t))}\rightarrow \infty\,.
\end{equation}
\qed 
\end{proof}

\begin{remark}
Away from the contact points at time $T_{max}$, the velocity field in the fluid stays smooth by elliptic regularity (by (\ref{reg1}) and (\ref{reg2})), so it is indeed next to the contact points that the blow-up is localized. 
\end{remark}

 We now establish that contact occurs with an infinite upward acceleration for the solid, except for the case where the contact zone at $T_{max}$ contains a curve of non zero length, in which case the acceleration becomes strictly positive as contact nears, while staying bounded.

\section{Positive or infinite upward solid acceleration at time of contact for the case without vorticity}
\label{section18}

From (\ref{1612.5bis}) we immediately have,
\begin{align}
2\f{E(0)+(\rho^f-\rho^s)g x^s_2|\Omega^s|}{(v^s_2)^2}\f{dv^s_2}{dt}
=&-\rho^f\int_{\p\Omega}\f{|u^f|^2}{2}\ n_2\ dl +(\rho^f-\rho^s)g |\Omega^s|\label{310817.1}\\
\ge &\rho^f\alpha_\Omega \int_{\p\Omega}\f{|u^f|^2}{2} dl +(\rho^f-\rho^s)g |\Omega^s|\n\\
&-\rho^f\underbrace{\int_{\p\Omega\cap\Gamma_1^c}\f{|u^f|^2}{2}(n_2+\alpha_{\Omega}) dl}_{\text{bounded by }\ (\ref{reg1})}\,,
\label{1612.5}
\end{align}
where we used in (\ref{1612.5}) the fact that $n_2<-\alpha_\Omega$ on $\Gamma_1$.

We can now conclude on our acceleration. We will have to distinguish two cases.

\noindent{\bf Case 1. $v^s_2(T_{max})<0$}

From (\ref{coef}), we infer from (\ref{1612.5}), the fact that the velocity stays away from zero, and our blow-up (\ref{100817.17}) that
\begin{equation}
\label{1612.6}
\lim_{t\rightarrow T_{max}^-}\f{dv^s_2}{dt}(t)=\infty\,.
\end{equation}

 From (\ref{1612.6}) and (\ref{ch-2}), we immediately have
\begin{equation}
\label{2812.1}
\lim_{t\rightarrow T_{max}^-}\int_{\p\ost} pn\ dl=\infty\,,
\end{equation}
which establishes the blow-up of the normalized (to zero on top of $\p\Omega$ on $x_1=0$) pressure on $\p\ost$ as we approach contact.

We now get back to (\ref{1612.5}). From the fact that the velocity stays away from zero as $t\rightarrow T_{\max}^-$ (from our assumption $v^s_2(T_{max})<0$) we then infer that $\f{1}{v^s_2(t)}$ stays bounded as $t\rightarrow T_{\max}^-$. Therefore, from (\ref{1612.5}) and $\f{d(v^s_2)^{-1}}{dt}=-\f{1}{(v^s_2)^2} \f{dv^s_2}{dt}$, this implies that
\begin{equation}
\label{3012.6}
\int_0^{T_{max}} \|u^f\|^2_{L^2(\p\Omega)}\ dt<\infty\,.
\end{equation}

\noindent{\bf Case 2. $v^s_2(T_{max})=0$}.
Here
\begin{equation}
\label{1303.3}
\lim_{t\rightarrow T_{max}^-} {v^s_2(t)}=0\,,
\end{equation}
simply translates into
\begin{equation}
\label{1303.4}
\lim_{t\rightarrow T_{max}^-} {(v^s_2(t))^{-1}}=-\infty\,.
\end{equation}
By integrating (\ref{310817.1}) from $0$ to $T_{max}$ we then obtain:
\begin{equation}
\label{1303.5}
\int_0^{T_{max}} \int_{\p\Omega} |u^f|^2 n_2\ dl\ dt=-\infty\,.
\end{equation}
By conservation of total energy, we have by (\ref{coef})
\begin{equation}
\rho^f \int_{\oft}|u^f|^2\ dx\rightarrow 2(E(0)-(\rho^s-\rho^f) g|\Omega^s| x_2(T_{max}))>0\,,\ \text{as}\ t\rightarrow T_{max}^-\,,
\end{equation}
Using (\ref{1303.1}), we obtain
\begin{equation}
\label{1303.6}
\int_{\p\ost}\phi n_2\ dx\ v^s_2(t)\rightarrow \f{2}{\rho^f} (E(0)-(\rho^s-\rho^f) g|\Omega^s| x_2(T_{max}))\,,\ \text{as}\ t\rightarrow T_{max}^-\,.
\end{equation}
Using (\ref{1303.2}) in (\ref{1303.6}) then yields:
\begin{equation}
\label{1303.7}
\int_{\p\Omega}\phi n_2\ dx\ v^s_2(t)\rightarrow -\f{2}{\rho^f}(E(0)-(\rho^s-\rho^f) g|\Omega^s| x_2(T_{max}))\,,\ \text{as}\ t\rightarrow T_{max}^-\,,
\end{equation}
where we also used the fact that $u^f$ is bounded in $L^2(\oft)$. By  a reasoning similar as when obtaining (\ref{1303.1}), this is equivalent to:
\begin{equation}
\label{1303.8}
\int_{\p\Omega} u^f\cdot\tau x_1\ dl\ v^s_2(t)\rightarrow \underbrace{\f{2}{\rho^f}(E(0)-(\rho^s-\rho^f) g|\Omega^s| x_2(T_{max}))}_{C_0>0}\,,\ \text{as}\ t\rightarrow T_{max}^-\,,
\end{equation}
Using Cauchy-Schwarz, this then provides some $C_1>0$ such that for $t$ close enough to $T_{max}$:
\begin{equation}
\label{1303.9}
-\int_{\p\Omega} |u^f|^2  dl \ge \f{C_1}{v^s_2(t)^2}\,.
\end{equation}
Reporting (\ref{1303.9}) in (\ref{1612.5}) yields:
\begin{equation}
\label{160817.1}
\f{dv^s_2}{dt}(t)\ge \f{\rho^f C_1\alpha_\Omega}{2 E(0)}=a_0>0\,,
\end{equation}
for any $t$ close enough to $T_{max}$, which again shows a positive upward acceleration for the rigid body as contact nears, opposing the fall.

We now prove that for the case when the part of  $\p\Omega$ intersecting $\p\Omega^s(T_{max})$ is of zero measure and $v^s_2(T_{max})=0$, we have an infinite upward acceleration for the solid at the time of contact. 

From now on $C$ denote a generic positive constant independent of $t<T_{max}$.

Let us now fix $\epsilon>0$. 

Using our assumption that the intersecting part of $\p\Omega$ is of zero length, we write 
\begin{equation}
\label{1303.10}
\p\Omega=\Gamma_\epsilon\cup(\Gamma_\epsilon^c\cap\p\Omega)\,,
\end{equation}
where $\Gamma_\epsilon$ is a union of curves containing the contact points at $T_{max}$ and whose total length is less than $\epsilon$.
We now write
\begin{equation}
\label{1303.11}
\left|\int_{\p\Omega} u^f\cdot\tau x_1\ dl\right|\le C \sqrt{\epsilon}\sqrt{-\int_{\Gamma_\epsilon} n_2 |u^f\cdot\tau|^2\ dl}+C \int_{\Gamma_\epsilon^c\cap\p\Omega} |u^f\cdot\tau|^2\ dl\,,
\end{equation}
since $n_2\le-\alpha_\Omega<0$ on the contact part of $\p\Omega$.

Due to our control of $u^f$ away from the contact zone by (\ref{reg1}), we have from (\ref{1303.11})
\begin{equation}
\label{1303.12}
\left|\int_{\p\Omega} u^f\cdot\tau x_1\ dl\right|\le C \sqrt{\epsilon}\sqrt{-\int_{\Gamma_\epsilon} n_2 |u^f\cdot\tau|^2\ dl}+C_\epsilon\,,
\end{equation}
where $C_\epsilon$ is independent of $t$ (but blows up as $\epsilon\rightarrow 0$).
With (\ref{1303.8}), (\ref{1303.12}) provides for $t$ close enough to $T_{max}$ (with $\epsilon>0$ fixed, and remembering that $\lim_{t\rightarrow T_{max}} v^s_2(t)=0$)):
\begin{equation}
\label{1303.13}
\left|\f{(E(0)-(\rho^s-\rho^f) g|\Omega^s| x^s_2(T_{max}))}{\rho^f v^s_2(t)}\right| \le C \sqrt{\epsilon}\sqrt{-\int_{\Gamma_\epsilon} n_2 |u^f\cdot\tau|^2\ dl}\,.
\end{equation}
 Therefore,
\begin{equation}
\label{1303.14}
\f{(E(0)-(\rho^s-\rho^f) g|\Omega^s| x^s_2(T_{max}))^2}{(\rho^f)^2 v^s_2(t)^2 C^2\epsilon} \le  \int_{\p\Omega} |u^f\cdot\tau|^2\ dl\,.
\end{equation}
Using (\ref{1303.14}) in (\ref{1612.5}) then yields for $t$ close enough to $T_{max}$:
\begin{equation*}
\f{dv^s_2}{dt}(t)\ge \f{(E(0)-(\rho^s-\rho^f) g|\Omega^s| x^s_2(T_{max}))}{4 \rho^f C^2}\f{\alpha_\Omega}{\epsilon}\,,
\end{equation*}
which given the arbitrary nature of $\epsilon>0$ provides:
\begin{equation}
\lim_{t\rightarrow T_{max}^-}\f{dv^s_2}{dt}(t)=\infty\,.
\end{equation}

We now treat the remaining case where the contact zone contains a curve $\Gamma_c\subset\Gamma_1\cap\p\Omega^s(T_{max})$ of non zero length. Since $n_2<-\alpha_\Omega$ on $\Gamma_1$, we have the existence of $f$ smooth such that $\Gamma_1$ is the graph of a function $f$ for $x_1\in\cup_{i\in I} [\alpha_i,\beta_i]$ for some $\alpha_i\le\beta_i$. In a neighborhood of the region of the rigid body which intersects $\Gamma_1$ at $T_{max}$, we also have that $\p\Omega^s(T_{max})$ is the graph of a function $g$, which equals $f$ on the contact zone. We have that $f=g$ for $x_1\in\cup_{i\in J} [a_i,b_i]$ ($J\subset I$ and $[a_i,b_i]\subset [\alpha_i,\beta_i]$) and for $\epsilon>0$ small, we have the existence of $c_\epsilon>0$ small such that
\begin{equation}
\label{200817.1}
\forall x\in [a_i-c_\epsilon,b_i+c_\epsilon]\,,\ f(x_1)\le g(x_1)\le f(x_1)+\epsilon\,.
\end{equation}
We now define for $t<T_{max}$ the distance in the vertical direction between the two curves at time $t$:
\begin{equation}
\label{180817.1}
\eta^s_2(t)=\int_{T_{max}}^t v^s_2(s) ds>0\,.
\end{equation}
We denote by (dropping the $i$ index)
\begin{align*}
S^-_{a,b,\epsilon}=&\{(x_1,f(x_1)); x_1\in [a-c_\epsilon,b+c_\epsilon]\}\subset\Gamma_1\,,\\
S^+_{a',b'}(t)=&\{(x_1,f(x_1)+\eta^s_2(t)); x_1\in [a',b']\subset[a,b]\}\,,\\
\Omega^f_{a,b,\epsilon}(t)=&\{(x_1,x_2); x_1\in [a-c_\epsilon,b+c_\epsilon]; x_2\in (f(x_1),g(x_1)+\eta^s_2(t))\}\,,\\
\Omega^f_{a',b'}(t)=&\{(x_1,x_2); x_1\in [a',b']\subset [a,b]; x_2\in (f(x_1),f(x_1)+\eta^s_2(t))\}\,.
\end{align*}

\begin{figure}[h]
\begin{tikzpicture}[scale=0.3]


 \draw[color=red,ultra thick] plot[smooth,tension=.6] coordinates{( (-7,9) (-5,7.8) (-2,6.2) (0,5) (2,4.1) (5,2.8)(10,1.2)};
  \draw[color=black,ultra thick] plot[smooth,tension=.6] coordinates{( -10,11)(-7,9)};
   \draw[color=black,ultra thick] plot[smooth,tension=.6] coordinates{( 10,1.2)(12,0.7)};
   \draw[color=blue,ultra thick] plot[smooth,tension=.6] coordinates{( (-7,4) (-5,2.8) (-2,1.2) (0,0) (2,-0.9) (5,-2.2)(10,-3.8)};
  \draw[color=black,ultra thick] plot[smooth,tension=.6] coordinates{( -10,5.5) (-8,4.5) (-7,4)};
   \draw[color=black,ultra thick] plot[smooth,tension=.6] coordinates{( 10,-3.8)(12,-4.5)};  
  \draw (0,3) node {\large $\Omega_{a,b}(t)$};
   \draw (8,4) node {\large $S^+_{a,b}(t)\subset\p\ost$};
   \draw (7,-4) node {\large $S^-_{a,b}(t)\subset\p\Omega$};
   \draw[color=purple, ultra thick] (-7,9)--(-7,4);
   \draw[color=purple, ultra thick] (10,1.2)--(10,-3.8); 
   \draw[color=green, ultra thick] (-8,9.6)--(-8,4.5);
   \draw[color=green, ultra thick] (11,1)--(11,-4.1); 
   \draw[thick,->] (-9,-9)--(14,-9);
   \draw (14,-10) node {\large $x_1$};
    \draw (-6.5,-9.5) node {\large $a$};
     \draw (10.5,-9.5) node {\large $b$};
     \draw (-8.5,-9.5) node { $a-c_\epsilon$};
     \draw (12.5,-8.5) node { $b+c_\epsilon$};
   \draw[black,dashed] (-7,4)-- (-7,-9);
   \draw[black,dashed] (10,-3.8)-- (10,-9);
   \draw[black,dashed] (-8,4.5)-- (-8,-9);
   \draw[black,dashed] (11,-4.1)-- (11,-9);
   \end{tikzpicture} 
 \caption{ Here, the blue and red curves are translated vertically from each other. Contact at time $T_{max}$ does not occur along the black curves. $\Omega^f_{a,b,\epsilon}(t)$ is defined as $\Omega^f_{a,b}(t)$ with the green vertical lines replacing the purple ones. }\label{fig3}  
\end{figure}
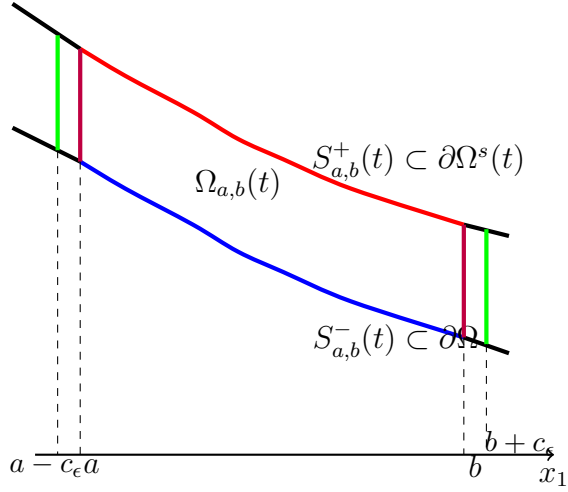

Since the fall is vertical with velocity constant in space, we have that
\begin{align*}
\Omega^f_{a,b}(t)\subset\Omega^f_{a,b,\epsilon}(t)\subset\Omega^f(t)\,,\\
S^+_{a,b}(t)\subset\p\Omega^s(t)\,,
\end{align*}
For $\alpha\in [0,\f{b-a}{4}]$ we now denote by 
\begin{equation}
\label{170817.3}
\Omega_\alpha(t)=\Omega^f_{a+\alpha,a+3\f{b-a}{4}+\alpha}(t)\subset\Omega^f(t)\,.
\end{equation}
From the divergence theorem, $\int_{\p\Omega_\alpha(t)} u^f\cdot n\ dl=0$, which provides if we denote $S_\alpha(t)=S^+_{a+\alpha,a+3\f{b-a}{4}+\alpha}(t)\subset\p\Omega^s(t)$:
\begin{equation}
\label{170817.5}
v^s_2(t) \int_{S_\alpha(t)} n_2 dl=\int_{\p\Omega_\alpha(t)\cap\{x_1=a+\alpha\}} u^f_1 dx_2 - \int_{\p\Omega_\alpha(t)\cap\{x_1=a+3\f{b-a}{4}+\alpha\}} u^f_1 dx_2\,.
\end{equation}
Integrating (\ref{170817.5}) with respect to $\alpha$ (variable $x_1$) between $0$ and $\f{b-a}{4}$ yields:
\begin{equation*}
v^s_2(t) \int_0^{\f{b-a}{4}}\int_{S_\alpha(t)} n_2 dl\ dx_1=\int_{\Omega^f_{a,a+\f{b-a}{4}}(t)} u^f_1 dx - \int_{\Omega^f_{a+3\f{b-a}{4},b}(t)}u^f_1 dx\,.
\end{equation*}
Therefore, by Cauchy-Schwarz applied to the right-hand side of this identity,
\begin{equation*}
|v^s_2(t)| \left|\int_0^{\f{b-a}{4}}\int_{S_\alpha(t)} n_2 dl\ dx_1\right|\le \sqrt{|\Omega^f_{a,b}(t)|} \|u^f\|_{L^2(\Omega^f_{a,b}(t))}\,,
\end{equation*}
which provides us (since $u^f$ is bounded in $L^2(\Omega^f(t))$ and $|\Omega^f_{a,b}(t)|=\eta^s_2(t)\times (b-a)$ ) with the existence of $C>0$ independent of $t<T_{max}$ such that
\begin{equation}
\label{200817.2}
|v^s_2(t)|\le C \sqrt{\eta^s_2(t)}\,.
\end{equation}

\begin{remark}
This inequality uses in a crucial way the fact contact occurs on a zone containing a curve of nonzero length.
\end{remark}

We define the vertical distance between the two graphs at time $t$ at $x_1$:
$$d(x_1,t)=g(x_1)-f(x_1)+\eta^s_2(t)\,,$$
and define for any $x_1\in[a-c_\epsilon,b+c_\epsilon]$ the vertical average of $u^f$:
\begin{equation*}
\bar{u}_1(x_1,t)=\f{1}{d(x_1,t)}\int_{f(x_1)}^{g(x_1)+\eta^s_2(t)} u^f_1(x_1,x_2,t)dx_2\,.
\end{equation*}
Due to $u^f=\nabla^\perp\phi$ with $\phi=0$ on $\p\Omega$ and $\phi=v^s_2(t) x_1$ on $\p\Omega^s(t)$, we have
\begin{equation}
\label{180817.3}
\bar{u}_1(x_1,t)=-\f{v^s_2(t)x_1}{d(x_1,t)}\,.
\end{equation}
Next since $\bar{u_1}(x_1,t)$ is a value taken by $u^f_1$ on the vertical segment $\{x_1\}\times [f(x_1),g(x_1)+\eta^s_2(t)]$, we have the existence of $\alpha(x_1,t)\in [f(x_1),g(x_1)+\eta^s_2(t)] $ such that $\bar{u}_1(x_1,t)=u^f_1(x_1,\alpha(x_1,t),t)$, which leads to
\begin{equation}
\label{190817.1}
 u_1(x_1,f(x_1),t)-\bar{u}_1(x_1,t)=\int_{\alpha(x_1,t)}^{f(x_1)} \f{\p u^f_1}{\p x_2}(x_1,x_2,t) dx_2\,.
\end{equation}
This implies by Cauchy-Schwarz that
\begin{equation}
\label{190817.2}
(u^f_1(x_1,f(x_1),t)-\bar{u}_1(x,t))^2\le d(x_1,t) \int_{f(x_1)}^{g(x_1)+\eta^s_2(t)} \left|\f{\p u^f_1}{\p x_2}(x_1,x_2,t)\right|^2 dx_2 \,.
\end{equation}
We now multiply (\ref{190817.2}) by the length element $\sqrt{1+f'^2(x_1)}$ on $\p\Omega$ and integrate the resulting relation with respect to $x_1\in [a-c_\epsilon,b+c_\epsilon]$. Remembering that $c_\epsilon$ was chosen so that
(\ref{200817.1}) was satisfied
we then obtain:
\begin{equation}
\label{190817.4}
\int_{S^-_{a,b,\epsilon}(t)} |u^f_1-\bar{u}_1|^2 dl\le C (\epsilon+\eta^s_2(t)) \int_{\Omega^f_{a,b,\epsilon}(t)} \left|\f{\p u^f_1}{\p x_2}(x_1,x_2,t)\right|^2 dx \,.
\end{equation} 
Using the triangular inequality we infer from (\ref{190817.4}) and (\ref{180817.3}) that
\begin{equation}
\label{180717.5}
\int_{S^-_{a,b,\epsilon}(t)} |u_1^f|^2 dl\le C (\epsilon+\eta^s_2(t)) \int_{\Omega^f_{a,b,\epsilon}(t)} |\f{\p u^f_1}{\p x_2}|^2 dx+ C \f{(v^s_2(t))^2}{\eta^s_2(t)^2} \,,
\end{equation}
where we remind $C$ is a generic constant independent of time.
 Since by our assumption on $\Gamma_1$, $n_2\le -\alpha_\Omega<0$ , and since $u^f\cdot n=0$ on $S^-_{a,b,\epsilon}$, we infer from (\ref{180717.5}) that
\begin{equation}
\label{190817.5}
\int_{S^-_{a,b,\epsilon}(t)} |u^f|^2 dl\le C (\epsilon+\eta^s_2(t)) \int_{\Omega^f_{a,b,\epsilon}(t)} |\f{\p u^f}{\p x_2}|^2 dx+ C \f{(v^s_2(t))^2}{\eta^s_2(t)^2} \,.
\end{equation}
We now work with $t$ close enough to $T_{max}$ so that $\eta^s_2(t)\le\epsilon$. Therefore, with our generic constant $C$, (\ref{190817.5}) becomes:
\begin{equation}
\label{190817.6}
\int_{S^-_{a,b,\epsilon}(t)} |u^f|^2 dl\le C \epsilon \int_{\Omega^f_{a,b,\epsilon}(t)} |\f{\p u^f}{\p x_2}|^2 dx+ C \f{(v^s_2(t))^2}{\eta^s_2(t)^2} \,.
\end{equation}
Summing over all regions of the type $S^-_{a,b,\epsilon}(t)$ in case contact occurs on a non connected set, we then have from (\ref{190817.6})
\begin{equation}
\label{190817.7}
\int_{\p\Omega} |u^f|^2 dl\le C \epsilon \int_{\Omega^f(t)} |\f{\p u^f}{\p x_2}|^2 dx+ C \f{(v^s_2(t))^2}{\eta^s_2(t)^2} +C_\epsilon\,,
\end{equation}
with $C_\epsilon$ being a constant independent of time (and becoming large as $\epsilon$ is small).
Due to (\ref{en2.b}), this inequality implies that
\begin{equation*}
\int_{\p\Omega} |u^f|^2 dl\le C \epsilon \int_{\p\Omega^f(t)} |u^f|^2 dl+ C \f{(v^s_2(t))^2}{\eta^s_2(t)^2} +C_\epsilon\,.
\end{equation*}
By choosing $\epsilon>0$ small enough, this inequality implies (we remind $C>0$ is generic and $\p\Omega^f(t)=\p\ost\cup\p\Omega$):
\begin{equation}
\label{190817.8}
\int_{\p\Omega} |u^f|^2 dl\le C \epsilon \int_{\p\Omega^s(t)} |u^f|^2 dl+ C \f{(v^s_2(t))^2}{\eta^s_2(t)^2} +C_\epsilon\,.
\end{equation}
Next we notice that since $n_2\le -\alpha_\Omega$ on $\Gamma_1$, we have $n_2\ge \alpha_\Omega>0$ in the region of $\p\Omega^s(t)$ near $\p\Omega$. Since $u^f$ is bounded away from the contact zone, we have
\begin{equation}
\label{190817.9}
\int_{\p\Omega^s(t)} |u^f|^2 dl\le \f{1}{\alpha_\Omega}\int_{\p\Omega^s(t)} |u^f|^2 n_2 dl +C\,.
\end{equation}
By integration by parts in $\Omega^f(t)$,
\begin{equation*}
\int_{\p\Omega^s(t)} |u^f|^2 n_2 dl=-\int_{\p\Omega} |u^f|^2 n_2 dl+2\int_{\Omega^f(t)}\f{\p u^f}{\p x_2}\cdot u^f\ dx\,.
\end{equation*}
Thus, for $\delta>0$ small to be precised later
\begin{align}
\int_{\p\Omega^s(t)} |u^f|^2 n_2 dl\le & \int_{\p\Omega} |u^f|^2 dl + \delta\int_{\Omega^f(t)}|\nabla u^f|^2 \ dx+\f{1}{\delta}\int_{\Omega^f(t)}|u^f|^2dx\n\\
\le &  \int_{\p\Omega} |u^f|^2 dl + \delta\int_{\Omega^f(t)}|\nabla u^f|^2 \ dx+\f{C}{\delta}\n\\
\le &  \int_{\p\Omega} |u^f|^2 dl + C \delta\int_{\p\Omega^f(t)}| u^f|^2 \ dl+\f{C}{\delta}\,,
\label{190817.11}
\end{align}
where we used (\ref{en2.b}). 
By using (\ref{190817.9}) we then see that for $C\delta<\f{\alpha_\Omega}{2}$ we have from (\ref{190817.11}):
\begin{equation*}
\int_{\p\Omega^s(t)} |u^f|^2 n_2 dl\le C \int_{\p\Omega} |u^f|^2 dl+C\,.
\end{equation*}
From (\ref{190817.9}), (and remembering $C$ is generic)
\begin{equation}
\label{190817.12}
\int_{\p\Omega^s(t)} |u^f|^2 dl\le C \int_{\p\Omega} |u^f|^2 dl+C\,.
\end{equation}
Therefore, by picking $\epsilon>0$ small enough, (\ref{190817.12}) used in (\ref{190817.8}) implies:
\begin{equation}
\label{190817.13}
\int_{\p\Omega} |u^f|^2 dl\le  C \f{(v^s_2(t))^2}{\eta^s_2(t)^2} +C\,.
\end{equation}
Using (\ref{190817.13}) in our formula for acceleration (\ref{1612.5}) we obtain:
\begin{equation*}
\f{dv^s_2}{dt}\le C \f{(v^s_2(t))^4}{\eta^s_2(t)^2} +Cv^s_2(t)^2\,.
\end{equation*}
Using (\ref{200817.2}) in the previous inequality provides:
\begin{equation*}
\f{dv^s_2}{dt}\le C \,,
\end{equation*}
and therefore with (\ref{160817.1})
\begin{equation*}
0<a_0\le\liminf_{t\rightarrow T_{max}^-} \f{dv^s_2}{dt}(t)\le \limsup_{t\rightarrow T_{max}^-} \f{dv^s_2}{dt}(t)<\infty\,.
\end{equation*}
which finishes the proof of Theorem \ref{theorem4}.
\qed

\end{document}